\pdfoutput=1
\RequirePackage{ifpdf}
\ifpdf % We~are running pdfTeX in pdf mode
\documentclass[pdftex]{sigma}
\else
\documentclass{sigma}
\fi

\numberwithin{equation}{section}

\newtheorem{Theorem}{Theorem}[section]
\newtheorem{Corollary}[Theorem]{Corollary}

\newtheorem{Proposition}[Theorem]{Proposition}
{ \theoremstyle{definition}

\newtheorem{Remark}[Theorem]{Remark} }

\def\C{\mathcal C}
\def\D{\mathcal D}
\def\E{\mathcal E}

\def\K{\mathcal K}
\def\I{\mathcal I}
\def\J{\mathcal J}
\def\L{\mathcal L}

\def\O{\mathcal O}

\def\X{\mathcal X}

\def\1{\mathbf 1}
\def\M{{\overline{\mathcal M}}}

\def\QQ{\mathbb Q}
\def\ZZ{\mathbb Z}
\def\CC{\mathbb C}

\def\Res{\operatorname{Res}}

\def\det{\operatorname{det}}

\def\tilde{\widetilde}
\def\p{\partial}
\def\a{\alpha}

\def\b{\beta}

\def\f{{\mathbf f}}

\def\t{{\mathbf t}}
\def\q{{\mathbf q}}

\def\lan{\langle}
\def\ran{\rangle}
\def\str{\operatorname{str}}

\def\ev{\operatorname{ev}}
\def\ft{\operatorname{ft}}

\def\Eu{\operatorname{Eu}}

\def\str{\operatorname{str}}
\def\Lie{\operatorname{Lie}}

\def\square{\Box}

\renewcommand{\Delta}{\triangle}

\begin{document}
%\allowdisplaybreaks

\newcommand{\arXivNumber}{2008.08182}

\renewcommand{\thefootnote}{}

\renewcommand{\PaperNumber}{018}

\FirstPageHeading

\ShortArticleName{Quantum K-Theory of Grassmannians and Non-Abelian Localization}

\ArticleName{Quantum K-Theory of Grassmannians \\ and Non-Abelian Localization\footnote{This paper is a~contribution to the Special Issue on Representation Theory and Integrable Systems in honor of Vitaly Tarasov on the 60th birthday and Alexander Varchenko on the 70th birthday. The full collection is available at \href{https://www.emis.de/journals/SIGMA/Tarasov-Varchenko.html}{https://www.emis.de/journals/SIGMA/Tarasov-Varchenko.html}}}

\Author{Alexander GIVENTAL and Xiaohan YAN}

\AuthorNameForHeading{A.~Givental and X.~Yan}

\Address{Department of Mathematics, University of California at Berkeley, Berkeley, CA 94720, USA}
\Email{\href{mailto:givental@math.berkeley.edu}{givental@math.berkeley.edu}, \href{mailto:xiaohan_yan@berkeley.edu}{xiaohan\_yan@berkeley.edu}}

\ArticleDates{Received August 25, 2020, in final form February 02, 2021; Published online February 26, 2021}

\Abstract{In the example of complex grassmannians, we demonstrate various techniques available for computing genus-0 K-theoretic GW-invariants of flag manifolds and more gene\-ral quiver varieties. In particular, we address explicit reconstruction of all such invariants using finite-difference operators, the role of the $q$-hypergeometric series arising in the context of quasimap compactifications of spaces of rational curves in such varieties, the theory of twisted GW-invariants including level structures, as well as the Jackson-type integrals playing the role of equivariant K-theoretic mirrors.}


\Classification{14N35}

\begin{flushright}
\begin{minipage}{75mm}
\em To Vitaly Tarasov and Alexander Varchenko, on their anniversaries
\end{minipage}
\end{flushright}

\renewcommand{\thefootnote}{\arabic{footnote}}
\setcounter{footnote}{0}

\section{Introduction}

Just as quantum cohomology theory deals with intersection numbers between interesting cyc\-les in moduli spaces of stable maps of holomorphic curves in a given target (say, a K\"ahler manifold),
quantum K-theory studies sheaf cohomology (e.g., in the form of holomorphic Euler characteristics) of interesting vector bundles over these moduli spaces. The beginnings of the subject can be traced back to the 20-year-old note~\cite{Givental:WDVV} by the first-named author, the foundational work by Y.-P.~Lee~\cite{Lee}, and their joint paper on complete flag manifolds, $q$-Toda lattices and quantum groups~\cite{Givental-Lee}. In recent years, however, the interest to quantum K-theory expanded due to more discoveries of its diverse relations with representation theory, integrable systems, and $q$-hypergeometric functions.

Apparently the interest was initiated by the 2012 preprint~\cite{MO} by D.~Maulik and A.~Okoun\-kov, who connected equivariant quantum cohomology of quiver varieties with R-matrices. In~2014, this led R.~Rim{\'a}nyi, V.~Tarasov and A.~Varchenko~\cite{RTV} to a conjectural description of the quantum K-ring of the {\em cotangent bundle} of a partial flag variety. In even more recent literature motivated by representation theory (see, e.g.,~\cite{KPSZ, Ok, PSZ}), certain $q$-hypergeometric series, interesting from the point of view of the theory of integrable systems, appeared as generating functions for K-theoretic Gromov--Witten (GW) invariants of {\em symplectic} quiver varieties. In this literature, K-theoretic computations are based, however, on the {\em quasimap} (rather than stable map) compactifications~\cite{CKM} of spaces of rational curves in the GIT quotients of linear spaces. Based on the experience with mirror symmetry and quantum K-theory of toric varieties~\cite{Givental:perm5} one anticipates the $q$-hypergeometric generating functions arising from quasimap spaces to never\-theless represent the ``genuine'' (i.e., based on stable map compactifications) K-theoretic GW invariants, yet such invariants of a different kind, or more complicated ones than naively expec\-ted. In any case, this brings up the question of comparison (first attempted by H.~Liu~\cite{Liu}) of~the two approaches.

In this paper, we examine in substantial detail the genus-0 quantum K-theory of grass\-mannians ${\rm Gr}_{n,N}(\CC)$. The gras\-sman\-ni\-ans can be described as the GIT quotients $\mathop{\rm Hom} \big(\CC^n,\CC^N\big)//$ ${\rm GL}_n(\CC)$, and are perhaps the simplest among homogeneous spaces or, more generally, quiver varieties outside the toric class. Most of our methods carry over to other quiver varieties (and all~-- to any partial flag manifolds), but we prefer to illustrate the available techniques by way of simplest representative examples, trading generality for simplicity of notation.

In Sections~\ref{Sec2} and \ref{Sec3} we show how the technique of fixed point localization in moduli spaces of~stable maps can be used in order to compute the so-called ``small J-function'' of the grassmannian~-- the generating function for simplest genus-0 K-theoretic GW-invariants of it.

In Section \ref{Sec4} we combine the same technique with the idea known as ``non-abelian localization''~\cite{BCK2} in order to prove the invariance of the genus-0 quantum K-theory of the grassmannian under a suitable infinite dimensional group of pseudo-finite-difference operators. A key point here (inspired by the appendix in the paper~\cite{Hori-Vafa} by K.~Hori and C.~Vafa) is to begin with the toric quotient $\mathop{\rm Hom}\big(\CC^n,\CC^N\big)//T^n=\big(\CC P^{N-1}\big)^n$ by the maximal torus $T^n\subset {\rm GL}_n(\CC)$, and use Weyl-group invariant finite-difference operators on~$n$ Novikov's variables of the toric manifold.
Just as in the case of toric manifolds~\cite{Givental:perm8}, this infinite dimensional group of symmetries is large enough in order to reconstruct ``all'' genus-0 invariants of the grassmannian from the small J-function.

In Section \ref{Sec5}, we address the aforementioned comparison problem by interpreting (in several somewhat different ways) the $q$-hypergeometric series arising from quasimap theory of the cotangent bundle spaces $T^*{\rm Gr}_{n,N}$ as certain ``genuine'' K-theoretic GW-invariants, and in particular show that, contrary to a naive belief articulated in the literature, these series fail to represent ``small'' J-functions (of anything).

In Section \ref{Sec6}, we apply the invariance result from Section \ref{Sec5} to illustrate the ``non-abelian quantum Lefschetz'' principle which characterizes genus-0 quantum K-theory of a complete intersection in (or a vector bundle space over) the grassmannian.

In Section \ref{Sec7}, we show how our techniques can be used to extend (to the case of grassmannians) the toric results obtained by Y.~Ruan and M.~Zhang~\cite{R-Z} about the {\em level structures} in quantum \mbox{K-theory}. As a by-product, we clarify (hopefully) the phenomenon of {\em level correspondence} between ``dual'' grassmannians ${\rm Gr}_{n,N}={\rm Gr}_{N-n,N}$ discovered recently by H.~Dong and Y.~Wen~\cite{Dong-Wen}.

In Section \ref{Sec8}, we exhibit a Jackson-type integral formula for the small J-function in the quantum K-theory of the grassmannian, inspired by the ``non-abelian localization'' framework from Section~\ref{Sec4}. Our logic is the same as in the aforementioned appendix~\cite{Hori-Vafa} by K.~Hori and C.~Vafa, where cohomological mirrors of ${\rm Gr}_{n,N}$ were proposed. However, our mirror formula looks different (and possibly new, see~\cite{M-R}) even in the cohomological GW-theory.

Namely, this cohomological mirror of the grassmannian has the form of complex oscillating integral
\[
\I:=\int_{\Gamma\subset \X_{Q}} {\rm e}^{\left(\sum_{ij} x_{ij}-\sum_{i\neq i'}y_{ii'}\right)/z} \frac{\bigwedge_{ij} {\rm d}\ln x_{ij} \bigwedge_{i\neq i'}{\rm d}y_{ii'}}{\bigwedge_i {\rm d}\ln \big(\prod_jx_{ij}/\prod_{i'}(y_{ii'}/y_{ii'})\big)}.
\]
Here $\X_{Q}$ is the complex torus in the linear space with coordinates $\{ x_{ij}\}$, $i=1,\dots,n$, $j=1,\dots, N$ and $\{ y_{ii'}\}$, $i,i'=1,\dots,n$, $i\neq i'$, given by $n$ equations
\[
\prod_j x_{ij}=Q \prod_{i'}(y_{ii'}/y_{i'i}), \qquad i=1,\dots n,
\]
and $\Gamma$ is a combination of Lefschetz thimbles in $\X_{Q}$, invariant under the Weyl group $S_n$ acting on the coordinates $\{ x_{ij} \}$, $\{y_{ii'} \}$ by simultaneous permutations of the indices $i$ and $i'$.

As a mirror symmetry test, let us examine the critical set of the phase function (``superpotential'') using Lagrange multipliers $p_1,\dots, p_n$:
\[
\sum_{i,j}x_{ij}-\sum_{i\neq i'}y_{ii'}-\sum_i p_i\bigg(\sum_j\ln x_{ij} -\sum_{i'\neq i} (\ln y_{ii'}-\ln y_{i'i})-\ln Q\bigg).
\]
The critical points are determined from
\[
x_{ij}=p_i,\qquad y_{ii'}=p_i-p_{i'},\qquad p_i^N+(-1)^nQ=0,
\]
where the third set of equations comes from the constraints. The algebra of functions on the critical set (which is a finite lattice $\ZZ_N^n\subset T^n$) invariant under permutations of $(p_1,\dots,p_n)$
is indeed isomorphic to the ``small'' quantum cohomology algebra of the grassmannian (as described by formula (3.39) in~\cite{Witten}).

\section[The small J-function of Gr {n,N}]{The small J-function of $\mathbf{Gr}_{\boldsymbol{n,N}}$}
\label{Sec2}

{\sloppy
Let $X:={\rm Gr}_{n,N}$ be the grassmannian of $n$-dimensional subspaces $V\subset \CC^N$. Its K-ring $K^0(X)$ is~generated by the exterior powers $\bigwedge^kV$ of the tautological bundle, $k=1,\dots,n$. Using the \mbox{splitting} principle, we will often write them as elementary symmetric functions $\sum_{1\leq i_1<\cdots<i_k\leq n}P_{i_1}\cdots P_{i_k}$ of K-theoretic Chern roots of $V=P_1+\cdots+P_n$.

}

\begin{Proposition} \label{prop1}
The K-theoretic Poincar{\'e} pairing on $K^0(X)$ is given by residue formula
\[
\chi(X; \Phi(P))= (-1)^n\Res_{P=1} \frac{\Phi(P) \ \prod_{i\neq j}(1-P_i/P_j)}{(1-P_1)^N\cdots(1-P_n)^N} \frac{{\rm d}P_1\wedge\cdots \wedge {\rm d}P_n}{P_1\cdots P_n},
\]
where $\Phi$ is any symmetric Laurent polynomial of $P_1,\dots,P_n$.
\end{Proposition}

The formula is obtained as the non-equivariant limit $\Lambda\to 1$ from its $T^N$-equivariant counterpart, where $T^N$ is the torus of diagonal matrices $\operatorname{diag}(\Lambda_1,\dots,\Lambda_N)$ acting on $\C^N$.

\begin{Proposition}\label{prop2}
The $T^N$-equivariant K-theoretic Poincar{\'e} pairing on $K^0_T(X)$ is given by
\[
\chi_T(X;\Phi(P,\Lambda))= \frac{(-1)^n}{n!} \Res_{P\neq 0,\infty} \frac{\Phi(P,\Lambda)\ \prod_{i\neq j}(1-P_i/P_j)}{\prod_{i=1}^n\prod_{j=1}^N(1-P_i/\Lambda_j)} \frac{{\rm d}P_1\wedge\cdots \wedge {\rm d}P_n}{P_1\cdots P_n}.
\]
\end{Proposition}
Here $\Phi$ is a Laurent polynomial in $P$ and $\Lambda$, symmetric in $P$, $\chi_T$ is the $T$-equivariant holomorphic Euler characteristic, taking values in the representation ring $\ZZ[\Lambda^{\pm}]$ of the torus, and the residue sum is taken over all poles $P_1=\Lambda_{i_1}$, \dots, $P_n=\Lambda_{i_n}$ (with this ordering of the equations, and $i_\a\neq i_\b$). The formula is proved by the direct application of Lefschetz' holomorphic fixed point formula.

The following $q$-hypergeometric series has emerged from a study of spaces of rational curves in the grassmannian based on their quasimap compactifications:
\[
J=\sum_{0\leq d_1,\dots,d_n}\frac{Q^{d_1+\cdots+ d_n}}{\prod_{i=1}^n\prod_{m=1}^{d_i}(1-q^mP_i)^N} \prod_{i,j=1}^n\frac{\prod_{m=-\infty}^{d_i-d_j}(1-q^mP_i/P_j)} {\prod_{m=-\infty}^0(1-q^mP_i/P_j)}.
\]

\begin{Remark}\label{rmk1}
The product on the right can be rearranged as
\[
\prod_{d_i>d_j} q^{\binom{d_i-d_j}{2}} \bigg(\!{-}\frac{P_i}{P_j}\bigg)^{d_i-d_j}\frac{P_j-q^{d_i-d_j}P_i}{P_j-P_i},
\]
and therefore may contain nilpotent factors $P_j-P_i$ in the denominator. It is not hard to see, however, that transposing $P_i$ and $P_j$ does not change the sum of terms with a fixed $d_1+\cdots +d_n$, implying that after clearing the denominators, the numerator becomes divisible by $P_j-P_i$ (namely, it changes sign under the transposition, and hence vanishes when $P_i=P_j$).
\end{Remark}

\begin{Remark}\label{rmk2}
Another consequence of the above rearrangement is that, with the exception of the term $Q^0$, the series consists of reduced rational functions of $q$. Namely, the factor at $Q^{d_1+\cdots+d_n}$ has no pole at $q=0$, and the $q$-degree of the denominator exceeds that of the numerator by
\[
N\sum_{i=1}^n\binom{d_i+1}{2}-\sum_{d_i>d_j}\binom{d_i-d_j+1}{2}\geq N-n+1\geq 2.
\]
\end{Remark}

\begin{Theorem}[{cf.~\cite{Taipale, Wen}}]\label{thm1}
The series $(1-q) J$ is the ``small J-function'' of the grassmannian ${\rm Gr}_{n,N}(\CC)$.
\end{Theorem}

Recall that the genus-0 quantum K-theory of a target space $X$, the ``big J-function'' is defi\-ned~as
\[
\t \mapsto \J(\t):=1-q+\t(q)+\sum_{d,m,\a} Q^d \phi_\a \bigg\lan \frac{\phi^\a}{1-qL_0},\t(L_1),\dots,\t(L_m)\bigg\ran_{0,m+1,d}^{S_m}.
\]
Here $Q$ is the Novikov's variable, $\{ \phi_\a \}$ and $\{ \phi^\a\}$ are Poincar{\'e}-dual bases in $K^0(X)$, $\t=\sum_k t_k q^k$ is a Laurent polynomial in $q$ with vector coefficients $t_k \in K^0(X)\otimes \QQ[[Q]]$, and the correlator represents the K-theoretic GW-invariant computes a suitable holomorphic Euler characteristic on the moduli spaces of stable maps $X_{g,m+1,d}:=\M_{g,m+1}(X,d)$ with the input (or ``insertion'') at the marked point (with the index $i=0,\dots,m$ in the above formula) of the form $\sum_k (\ev_i^*t_k) L_i^k$, where $L_i$ stands for the universal cotangent line bundle at the $i$th marked point. From among several flavors of such K-theoretic GW-invariants (ordinary as in~\cite{Givental-Tonita}, or permutation-equivariant as in~\cite{Givental:perm9}), we will currently use the {\em permutation-invariant} ones (as the superscript $S_m$ indicates), i.e., computing the super-dimension of the part of the sheaf cohomology on the moduli space $X_{0,m+1,d}$ which is {\em invariant} under permutations of the $m$ marked points with the indices $i=1,\dots,m$ carrying the symmetric inputs~$\t(L_i)$.

The ``small J-function'' is obtained from $\J$ by setting the input $\t=0$. In particular, this eliminates the role of the permutation group, and so $\J(0)$ represents the ``ordinary'' K-theoretic GW-invariants. Thus, according to Theorem~\ref{thm1},
\[
\J(0):=(1-q)+\sum_{d,\a}Q^d\phi_\a\bigg\lan\frac{\phi^\a}{1-qL_0}\bigg\ran_{0,1,d}=(1-q) J.
\]
Theorem~\ref{thm1} is obtained as the non-equivariant limit $\Lambda\to 1$ from the following result about the $T^N$-equivariant version $\J^T$ ``big J-function'' of the grassmannian.

\begin{Theorem}[{cf.~\cite{Taipale, Wen}}]\label{thm2}
$\J^T(0)=(1-q)J^T$, where
\[
J^T=\sum_{0\leq d_1,\dots,d_n}\frac{Q^{d_1+\cdots+d_n}}{\prod_{i=1}^n\prod_{j=1}^N\prod_{m=1}^{d_i}(1-q^mP_i/\Lambda_j)} \prod_{i,j=1}^n\frac{\prod_{m=-\infty}^{d_i-d_j}(1-q^mP_i/P_j)} {\prod_{m=-\infty}^0(1-q^mP_i/P_j)}.
\]
\end{Theorem}

Note that by the very definition (the same as for $\J$ with the correlators taking values in the representation ring $\ZZ[\Lambda^{\pm}]$), the function $\J^T$ is a $Q$-series with coefficients which are rational functions of $q$ with vector values in $K:=K^0_T(X)\otimes \QQ[[Q]]$. Abusing the language we call such series rational functions of $q$, denote the space they form by $\K:=K(q^{\pm})$, and
call it the {\em loop space}. The part $(1-q)+\t(q)$ (``dilaton shift''$+$''input'') belongs to the subspace $\K_{+}$ consisting of Laurent polynomials (they can have poles only at $q=0,\infty$), while the sum of the correlators belongs (as it is not hard to see) to the complementary subspace $\K_{-}=\{ \f \in \K\ | \ \f(0)\neq \infty, \f(\infty)=0.\}$. It follows from Remark~\ref{rmk2} above (which applies to $J^T$ as well) that $(1-q)J^T\equiv 1-q\ \mod \K_{-}$. Thus, the non-obvious statement of Theorem~\ref{thm2} is that $(1-q)J^T$ represents a value of $\J^T$ at all.

The technique of fixed point localization we intend to use goes back to paper~\cite{Jeff_Brown} by
J.~Brown, and was adapted to the K-theoretic situation in~\cite{Givental:perm2}. The technique, applicable whenever the target carries a torus action with isolated fixed points and isolated $1$-dimensional orbits, completely characterizes all values of the big J-function $\J^T$ as the set of those rational functions $\f\in \K$ which pass two tests: criterions (i) and (ii). They are formulated in terms of specializations $\f_{\a}$ of $\f$ to the fixed point of the torus action in $X$. In the case of the grassmannian, take for example the fixed point $V_{1,\dots,n}=\mathop{\rm Span}(e_1,\dots,e_n)$ where (we may assume by choosing the ordering) $(P_1,\dots,P_n)=(\Lambda_1,\dots,\Lambda_n)$:
\[
J^T_{(1,\dots,n)}= \sum_{0\leq d_1,\dots,d_n} \frac{Q^{d_1+\cdots+d_n}}{\prod_{i=1}^n\prod_{j=1}^N\prod_{m=1}^{d_i}(1-q^m\Lambda_i/\Lambda_j)} \prod_{i,j=1}^n\frac{\prod_{m=-\infty}^{d_i-d_j}(1-q^m\Lambda_i/\Lambda_j)} {\prod_{m=-\infty}^0(1-q^m\Lambda_i/\Lambda_j)}.
\]
Note that the 1st factor contains the product (coming from $j=i$):
\[
\frac{1}{\prod_{i=1}^n \prod_{m=1}^{d_i}(1-q^m)}
\]
with poles at roots of unity, while all other poles are elsewhere (at $q=(\Lambda_i/\Lambda_j)^{-1/m}$). Each term of the series considered as rational functions of $q$ can be split (e.g., using partial fraction decomposition) into the sum of a reduced rational function with poles at the roots of unity and a rational function with poles elsewhere. The result will be interpreted (or rather termed) as~a~mero\-mor\-phic function in a neighborhood of the roots of unity.

Criterion (i) stipulates that {\em $\f_\a$, when interpreted as a meromorphic function in a neighborhood of the roots of unity, must represent a value (over a suitable ground ring) of the big J-function of the point target space.} We will return to this criterion in the next section and explain how it~can be verified.

Criterion (ii) controls residues of $\f_\a(q){\rm d}q/q$ at the poles originating from $T$-equivariant covers of $1$-dimensional orbits. Namely, the tangent space to the grassmannian at the fixed point $V_{(1,\dots,n)}$ carries the torus action with the distinct eigenvalues $\Lambda_j/\Lambda_i$, $i=1,\dots,n$, $j=n+1,\dots,N$. Consequently, for each choice of $i$ and $j$ there is a $1$-dimensional orbit, which
compactifies into~$\CC P^1$ connecting this fixed point with another one. For instance, taking $i=1$ and $j=n+1$, we find such an orbit connecting $V_{(1,\dots,n)}$ with $V_{(2,\dots,n+1)}$. Let $\phi\colon \CC P^1\to \CC P^1$ be the map $z\mapsto z^{m_0}$ ramified at $z=0,\infty$ (representing the two fixed points which we call $\a$ and $\b$). Criterion (ii) has the form of the recursion relation:
\[
\Res_{q=(\Lambda_j/\Lambda_i)^{1/m_0}}\f_\a(q) \frac{{\rm d}q}{q} = -\frac{Q^{m_0}}{m_0}
\frac{\Eu(T_\a X)}{\Eu(T_{\phi}X_{0,2,m_0})} \, \f_{\b}\Big|_{q=(\Lambda_j/\Lambda_i)^{1/m_0}},
\]
where $\Eu$ are equivariant K-theoretic Euler classes: of the tangent space to $X$ at $\a$, and to the moduli space of degree-$m_0$ stable maps with 2 marked points at the point represented by the $m_0$-fold cover $\phi$ respectively.

We compute $\Res_{q=(\Lambda_{n+1}/\Lambda_1)^{1/m_0}} (1-q)J^T_{(1,\dots,n)}(q)dq/q$, replacing $d_1$ with $d_1+m_0$, assuming $d_i=d_1$ when $i=n+1$, and using $x:= (\Lambda_{n+1}/\Lambda_1)^{1/m_0}$:
\begin{gather*}
\Res_{q=x} (1-q)J^T_{(1,\dots,n)}(q)\frac{{\rm d}q}{q} = -(1-x)\frac{Q^{m_0}}{m_0}
\\ \qquad
{}\times\frac{1}{\prod_{j=1}^N\prod_{m=1}^{m_0}\,_{(j,m)\neq (n+1,m_0)}\ (1-x^m\Lambda_1/\Lambda_j)} \prod_{j=2}^n\prod_{m=1}^{m_0}\frac{1-x^m\Lambda_1/\Lambda_j}{1-x^m\Lambda_j/\Lambda_{n+1}}
\\ \qquad
{}\times\sum_{0\leq d_1,\dots,d_n}\frac{Q^{d_1+\cdots+d_n}}{\prod_{i=2}^{n+1}\prod_{j=1}^N\prod_{m=1}^{d_i}
(1-x^m\Lambda_i/\Lambda_j)} \prod_{i,j=2}^{n+1}\frac{\prod_{m=-\infty}^{d_i-d_j}
(1-x^m\Lambda_i/\Lambda_j)}{\prod_{m=-\infty}^0(1-x^m\Lambda_i/\Lambda_j)}.
\end{gather*}
The sum together with the factor $1-x$ yields $(1-q)J^T_{(2,\dots,n+1)}(q)|_{q=x}$, so it remains to interpret the recursion coefficient in terms of the Euler classes.

Applying Lefschetz' fixed point formula on ${\rm Gr}_{n,N}(\CC)$, we've already used that $\Eu(T_{(1,\dots,n)}X)\!=\prod_{i=1}^n\prod_{j=n+1}^N(1-\Lambda_i/\Lambda_j)$. In order to compute $\Eu(T_{\phi}X_{0,2,m_0})$, we note that the $1$-dimensional orbit connecting the fixed points $\mathop{\rm Span} (e_1,\dots,e_n)$ and $\mathop{\rm Span} (e_2,\dots,e_{n+1})$ consists of subspaces $V_t:=\mathop{\rm Span}(t e_1+(1-t) e_{n+1}, e_2,\dots,e_{n-1})$. Consequently, restricted to $\CC P^1=\{ V_t \}$, the tangent bundle to the grassmannian, which has the form $\mathop{\rm Hom} (V_t,\CC^N/V_t)$, can be described in terms of~the~Hopf bundle $L$ over $\CC P^1$ and its complementary $L':=\mathop{\rm Span}(e_1,e_{n+1})/L=\Lambda_1\Lambda_{n+1}L^{-1}$ as
\[
\mathop{\rm Hom} \big(L \oplus \mathop{\rm Span}(e_2,\dots,e_n), L' \oplus \mathop{\rm Span}(e_{n+2},\dots,e_N)\big).
\]
On the $m_0$-fold cover $\phi\colon \CC P^1\to \CC P^1$, the contributions to the Euler class of the $T$-modules $H^0\big(\CC P^1; \phi^*L^{-1}\otimes \mathop{\rm Span} (e_j)\big)$ and $H^0\big(\CC P^1; \phi^*L'\otimes \mathop{\rm Span}(e_i)^{-1}\big)$ are respectively
\[
\prod_{m=0}^{m_0}(1-x^m\Lambda_1/\Lambda_j)\qquad \text{and}\qquad \prod_{m=0}^{m_0}(1-x^m\Lambda_i/\Lambda_{n+1}).
\]
The contribution of $H^0\big(\CC P^1; \phi^*L^{-1}\otimes L'\big)$ is (as in the case of $X=\CC P^1$) $\prod_{m=-m_0, m\neq 0}^{m_0}(1-x^m)$. Combing all the contributing factors, we find
\begin{gather*}
\frac{\Eu(T_{(1,\dots,n)}X)}{\Eu(T_{\phi}X_{0,2,m_0})}=
\frac{\prod_{i=1}^n\prod_{j=n+1}^N(1-\Lambda_i/\Lambda_j)}
{\prod_{i=2}^n\prod_{j=n+2}^N(1-\Lambda_i/\Lambda_j)}
\frac{1}{\prod_{m=-m_0, m\neq 0}^{m_0}(1-x^m)}
\\ \hphantom{\frac{\Eu(T_{(1,\dots,n)}X)}{\Eu(T_{\phi}X_{0,2,m_0})}=}
{}\times \frac{1}{\prod_{j=n+2}^N\prod_{m=0}^{m_0}(1-x^m\Lambda_1/\Lambda_j)}
\frac{1}{\prod_{i=2}^n\prod_{m=0}^{m_0}(1-x^m\Lambda_i/\Lambda_{n+1})}.
\end{gather*}
Checking that this expression matches exactly the recursion coefficient for $J^T$ (the middle line in~the formula for the residue) is the matter of a straightforward (though somewhat cumbersome) rearrangement of the factors.

\begin{Remark}\label{rmk3}
Note that the structure of the recursion relations (ii) and the values of the recursion coefficients completely characterize the big J-function of a particular theory, since criterion (i) does not involve any additional choices.
\end{Remark}

\section{Quantum K-theory of the point}
\label{Sec3}

Genus-0 permutation-equivariant K-theoretic GW-invariants of the point are represented by the ``big J-function'' of the form
\[
\J_{pt}(\t):=(1-q)+\t(q)+\sum_{m=2}^{\infty}\chi
\bigg(\M_{0,m+1}/S_m; \frac{1}{1-L_0q}\otimes_{i=1}^m \t(L_i)\bigg).
\]
Here $L_i$ are the universal cotangent line bundles over the Deligne--Mumford spaces $\M_{0,m+1}$. The~holomorphic Euler characteristic $\chi$ on the orbifold $\M_{0,m+1}/S_m$ computes the super-dimen\-sion of the $S_m$-invariant part of the sheaf cohomology on $\M_{0,m+1}$. The insertions $\t(L_i)$ (and hence the input $\t(q)$) can be in fact any rational functions of $L_i$ (respectively of $q$) as long as they don't have poles at roots of unity. On the contrary, each $\chi$-term is a reduced rational function of $q$ with poles only at the roots of unity (of order $\leq m$). Both this and the previous claim easily follow from the general structure of the Lefschetz fixed point formula (applied on $\M_{0,m+1}/S_m$).

As it is explained in~\cite{Givental:perm1, Givental:perm9}, the action of permutation groups on the sheaf cohomology is captured by the above $S_m$-invariants taking values in an arbitrary $\lambda$-algebra, i.e., a ring $R$ equipped with the Adams operations $\Psi^k\colon R\to R$. E.g., for $\tau\in R$
\[
\chi\big(\M_{0,m}/S_m;\otimes_{i=1}^m \big(\tau L_i^d\big)\big) := \frac{1}{m!} \sum_{h\in S_m}\prod_{k=1}^{\infty} (\Psi^k\tau)^{l_k(h)} \str_h H^*\big(\M_{0,m};\otimes_{i=1}^mL_i^d\big),
\]
where $l_k(h)$ denotes the number of cycles of length $r$ in the cycle decomposition of permutation $h$. For mode detail, we refer to~\cite{Givental:perm1} or~\cite{Givental:perm9}, where this example is
extrapolated to general $R$-valued insertions $\t(L_i)$. Note that only $\Psi^r$ with $r>0$ are used in this definition.

We assume that $R$ is complete in the adic topology defined by a certain ideal $R_{+}\subset R$, which is respected by the Adams operations $\Psi^k$ with $k>0$ in the sense that $\Psi^k(R_{+})\subset R_{+}^k$, and that the input $\t$ is ``small'' in the sense that it takes values in $R_{+}$. In fact the latter property guarantees that the $m$-th $\chi$-term of the series $\J_{pt}$ takes values in $R_{+}^m$, with assures $R_{+}$-adic convergence of the series.

The genus-$0$ permutation-equivariant GW-invariants of the point target space are completely described in~\cite{Givental:perm3}. Namely, given a ground $\lambda$-algebra $R$, the range of the big J-function
$\t\mapsto \J_{pt}(\t)$, which is a semi-infinite cone (that we will denote $\L_{pt}$) in the (completed) space of $R$-valued rational functions of $q$ (which we will denote $R(q^{\pm})$) is explicitly parameterized as
\[
(1-q) {\rm e}^{ \sum_{k>0}\Psi^k(\tau)/k(1-q^k)}\big(1+R_{+}[q^{\pm}]\big).
\]
Here $\tau\in R_{+}$, and the notation $R_{+}[q^{\pm}]$ is reserved for the completion in the $R_{+}$-adic topology of~the~space of rational functions of~$q$ which have no poles at roots of unity and take values in~$R_{+}$.

In our arguments, we will take advantage of the possibility to replace one ground $\lambda$-algebra with another related to it by a homomorphism respecting the Adams operations. In simple terms: If some
$\f\in R(q^{\pm})$ is known to lie in $\L_{pt}$ (over a given ground ring~$R$), i.e., the part $\f_{-}$ with poles at the roots of unity represents K-theoretic GW-invariants with the input defined by~the part $\f_{+}$ with poles away from roots of unity, the same will remain true when the values of some parameters (coordinates on $\operatorname{Spec} R$) are specialized in a way commuting with $\Psi^k$ for all~$k>0$.

Another important property of the cone $\L_{pt}$ that we will rely on is its invariance under a~cer\-tain group of (pseudo) finite-difference operators. Namely, let $R=\QQ[[Q]]$, where $\Psi^kQ=Q^{|k|}$, $k=\pm 1,\pm 2,\dots$, and let $D(q^{Q\p_Q},Q)$ be a finite difference operator (which we should assume ``small'' in $R_{+}$-adic sense to assure convergence). It is almost obvious that the linear vector field $\f \mapsto D\f$ in $R(q^{\pm})$ is tangent to $\L_{pt}$, and therefore ${\rm e}^D$ preserves $\L_{pt}$. Moreover, according to a~result from~\cite{Givental:perm4}, $\L_{pt}$ is preserved by the operator
\[
\f\mapsto {\rm e}^{ \sum_{k>0} \Psi^k(D(q^{kQ\p_Q},Q))/k(1-q^k)} \f.
\]

Our goal in this section is to verify that the series $J^T_{(1,\dots,n)}$ from the previous section satis\-fies criterion~(i), i.e., that it represents a value of $\J_{pt}$ when interpreted as a meromorphic function of $q$ in a neighborhood of roots of unity. For this, we begin with the ground ring $R=\QQ\big[\Lambda_1^{\pm},\dots,\Lambda_N^{\pm}\big][[Q_1,\dots,Q_n]]$, with the Adams operations acting by $\Psi^k\big(\Lambda_j^{\pm 1}\big)=\Lambda_j^{\pm k}$, $\Psi^k(Q_i)=Q_i^{|k|}$, and set $R_{+}=(Q_1,\dots,Q_n)$. In particular, taking $\tau=Q_1+\cdots+Q_n$, we~find that $\L_{pt}\ni (1-q) J_{pt}$, where $J_{pt}$ is the following product of $q$-exponential functions:
\begin{gather*}
J_{pt}:= {\rm e}^{\sum_{i=1}^n\sum_{k>0}Q_i^k/k(1-q^k)}
= \prod_{i=1}^n \sum_{d_i=0}^{\infty}\frac{Q_i^{d_i}}{\prod_{m=1}^{d_i}(1-q^m)}
 = \sum_{0\leq d_1,\dots,d_n}\frac{Q_1^{d_1}\cdots Q_n^{d_n}}{\prod_{i=1}^n\prod_{m=1}^{d_i}(1-q^m)}.
 \end{gather*}

Following~\cite{Givental:perm4}, we are going to use $Q$-independent finite-difference operators of the form $D_{l,\Lambda}:=-q \Lambda (1-q^{l\cdot Q\p_Q})$, where $l\cdot Q\p_Q:=\sum_i l_iQ_i\p_{Q_i}$, and $\Lambda$ is a formal variable added to the ground ring, $\Psi^k\Lambda:=\Lambda^{|k|}$. The corresponding $\L_{pt}$-preserving operator is
\[
\Gamma_{l,\Lambda}:= {\rm e}^{ -\sum_{k>0}\Lambda^k
(1-q^{kl\cdot Q\p_Q})q^k/k(1-q^k)}.
\]
This expression is in fact the asymptotical expansion (near the unit circle on the $q$-plane) of the ratio:
\[
\prod_{m=-\infty}^0\big(1-\Lambda q^{l\cdot Q\p_Q}q^m\big)/ \prod_{m=-\infty}^0(1-\Lambda q^m).
\]
The ratio and its asymptotical expansion act on monomials $Q^d=Q_1^{d_1}\cdots Q_n^{d_n}$ the same way:
\[
\Gamma_{l,\Lambda}Q^d= Q^d\ \frac{\prod_{m=-\infty}^0(1-\Lambda q^{m+(l\cdot d)})}{\prod_{m=-\infty}^0(1-\Lambda q^m)} = Q^d\ \frac{\prod_{m=-\infty}^{l\cdot d}(1-\Lambda q^m)}{\prod_{m=-\infty}^0(1-\Lambda q^m)}.
\]
Note that the right hand side is a rational function of $q$ with poles away from roots of unity. Considering $\Gamma_{l,\Lambda}(1-q)J_{pt}$ as a point of $\L_{pt}$ over the ground ring $R[[\Lambda]]$, we conclude therefore, that being a $Q$-series with coefficients which are rational function of $q$ with coefficients polynomial in $\Lambda$, it is a point of $\L_{pt}$ over $R[\Lambda]$. Thus, it will turn into a point of $\L_{pt}$ over $R$ when $\Lambda$ is replaced by any non-trivial monomial from $\QQ[\Lambda_1^{\pm},\dots,\Lambda_N^{\pm}]$, e.g., by $\Lambda_i/\Lambda_j$ with $i\neq j$.

In order to complete our fixed point localization proof of Theorem~\ref{thm2}, we apply to $J_{pt}$ the following operators (where $l=\1_i$ contains $1$ in the $i$-th position and $0$ everywhere else):
\begin{gather*}
\bigg(\prod_{i=1}^n\prod_{j=1, j\neq i}^N \Gamma^{-1}_{\1_i,\Lambda_i/\Lambda_j}\bigg) \bigg(\prod_{i,j=1}^n\Gamma_{\1_i-\1_j,\Lambda_i/\Lambda_j}\bigg)\ J_{pt}
\\ \qquad
{}=\sum_{0<d_1,\dots,d_n}\frac{Q_1^{d_1}\cdots Q_n^{d_n}}{\prod_{i=1}^n\prod_{j=1}^N\prod_{m=1}^{d_i}(1-q^m\Lambda_i/\Lambda_j)} \prod_{i,j=1}^n\frac{\prod_{m=-\infty}^{d_i-d_j}(1-q^m\Lambda_i/\Lambda_j)}{\prod_{m=-\infty}^0(1-q^m\Lambda_i/\Lambda_j)}.
\end{gather*}
The terms of the last sum are interpreted as meromorphic functions in the neighborhood of~roots of unity, i.e., with poles (which come from the factors with $i=j$ in the left product) at the roots of unity only.

When multiplied by $1-q$, the latter series lies in $\L_{pt}$ over the ground $\lambda$-algebra $R=\QQ[\Lambda^{\pm}][[Q_1,\dots,Q_n]]$. The substitution $Q_1=\cdots=Q_n=Q$ (which induces a homomorphism of~$\lambda$-algebras $R\to R_0:=\QQ[\Lambda^{\pm}][[Q]]$) yields therefore a series which lies in the range $\J_{pt}$ over~$R_0$. It actually coincides with the localization $(1-q) J^T_{(1,\dots,n)}$ of $J^T$ at the indicated fixed point. Therefore $(1-q) J^T_{(1,\dots,n)}$, when interpreted as a series of meromorphic functions near roots of~unity, satisfies criterion (i). Due to the Weyl group symmetry between all fixed points, and between all $1$-dimensional orbits connecting them, this finishes the proof of Theorem~\ref{thm2}.

\section{Non-abelian localization and explicit reconstruction}
\label{Sec4}

The approach to computing GW-invariants of GIT quotients via non-abelian localization (and eventually quasimap compactifications) was proposed by A.~Bertram, I.~Ciocan-Fontanine and B.~Kim~\cite{BCK2} following their proof~\cite{BCK1} of the Hori--Vafa conjecture. The conjecture (formulated in the appendix to~\cite{Hori-Vafa}) gave a novel proposal for the mirrors of GIT quotients $\CC^M//G$. The idea,
illustrated by K.~Hori and C.~Vafa in the example of the grassmannians, was to replace the factorization by a semi-simple $G$ with the (cohomologically equivalent to it) succession of the factorizations by its maximal torus $T$ and then by its Weyl group $W=N(T)/T$. The first step yields a toric manifold, whose mirror and genus-$0$ GW-invariants are well-understood. In the case of the grassmannian ${\rm Gr}_{n,N}=\mathop{\rm Hom}(\CC^n,\CC^N)//{\rm GL}_n(\CC)$, it is the product $\tilde{X}:=(\CC P^{N-1})^n$ of~projective spaces. Its small $T^N$-equivariant (K-theoretic) J-function is $(1-q) J_{\tilde{X}}^T$, where
\[
J^T_{\tilde{X}}=\sum_{d_1,\dots,d_n\geq 0}\frac{Q_1^{d_1}\cdots Q_n^{d_n}}{\prod_{j=1}^N\prod_{i=1}^n\prod_{m=1}^{d_i}(1-q^mP_i/\Lambda_j)},
\]
where $P_i$ are the Hopf bundles over the factors. The second step can be described this way:
\[
J^T=J^T_{\Pi \mathfrak{g}/\mathfrak{t}}|_{Q_1=\cdots=Q_n=Q, \Lambda_0=1}, \qquad \text{where}\quad J^T_{\Pi \mathfrak{g}/\mathfrak{t}}:=\prod_{i\neq j} \Gamma_{\1_i-\1_j,\,\Lambda_0 P_i/P_j} \,J^T_{\tilde{X}}.
\]
The $\Gamma$-operators here,
\[
\Gamma_{\1_i-\1_j, \Lambda_0 P_i/P_j}={\rm e}^{ -\sum_{k>0} (\Lambda_0P_i/P_j)^k\big(1-q^{kQ_i\p_{Q_i}-kQ_j\p_{Q_j}}\big)q^k/k(1-q^k)},
\]
correspond to the roots of $\mathfrak{g}$, i.e., in our case of $\mathfrak{g}=\mathfrak{gl}_n(\CC)$ to the line bundles $P_i/P_j$ for $i\neq j$. Explicitly
\[
J^T_{\Pi \mathfrak{g}/\mathfrak{t}}=\sum_{d_1,\dots,d_n\geq 0}\frac{Q_1^{d_1}\cdots Q_n^{d_n}}{\prod_{j=1}^N\prod_{i=1}^n\prod_{m=1}^{d_i}(1-q^mP_i/\Lambda_j)}
\prod_{i\neq j}\frac{\prod_{m=-\infty}^{d_i-d_j}(1-q^m\Lambda_0P_i/P_j)}{\prod_{m=-\infty}^0
(1-q^m\Lambda_0P_i/P_j)}.
\]
According to~\cite{Givental:perm5}, $(1-q)J^T_{\Pi {\mathfrak{g}/\mathfrak{t}}}$ represents a value of the big J-function of the {\em super-space} $\Pi E$, where $E$ is a vector bundle over $\tilde{X}$, equal to $\oplus_{i\neq j}P_j/P_i$ in the case at hands, which is associated with the adjoint action of the maximal torus on $\mathfrak{g}/\mathfrak{t}$, and $\Pi$ indicates the parity change of the fibers. By definition, the quantum K-theory of such a super-space is obtained by systematically replacing the virtual structure sheaves $\O_{g,m,d}$ of the moduli spaces $\tilde{X}_{g,m,d}$ with $\O_{g,m,d}\otimes \Eu_{\CC^{\times}}(\ft_*\ev^*E)$, where the subscript in the K-theoretic Euler class indicates it is equi\-variant with respect to the scalar action of $\Lambda_0\in \CC^{\times}$ on the fibers of $E$, and $\ft\colon \tilde{X}_{g,m+1,d}\to \tilde{X}_{g,m,d}$ and $\ev\colon \tilde{X}_{g,m+1,d}\to \tilde{X}$ are respectively the forgetting of and evaluation at the last marked point.

On the other hand, the {\em explicit reconstruction} results of~\cite{Givental:perm8} tell us how to parameterize the entire big J-function of a toric manifold (or super-manifold) from one value of it. Namely, the range of the big J-function, $\L_{\Pi \mathfrak{g}/\mathfrak{t}}$ in our example, is invariant under the action of a huge group,~${\mathcal P}$, of pseudo-finite-difference operators in Novikov's variables $Q_1,\dots,Q_n$. It is generated by the exponentials ${\rm e}^D$ of any ($R_+$-adically small) finite-difference operators $D(Pq^{Q_1\p_Q}, Q, q)$, and by operators of the form\footnote{Note that above operators $\Gamma_{\1_i-\1_j,\Lambda_0 P_i/P_j}$ are the compositions of the operators of multiplication by ${\rm e}^{\sum_{k>0}\Psi^k(\Lambda_0 P_i/P_j)/k(1-q^{-k})}$, whose cumulative effect, according to the Adams--Riemann--Roch (see~\cite{Givental:perm10}), is to transform $\L_{\tilde{X}}$ to $\L_{\Pi \mathfrak{g}/\mathfrak{t}}$, and of the operators of this form with $D= q\Lambda_0 P_iq^{Q_i\p_{Q_i}}/P_jq^{Q_j\p_{Q_j}}$ which preserve $\L_{\Pi \mathfrak{g}/\mathfrak{t}}$.}
\[
{\rm e}^{ \sum_{k>0} \Psi^k\big(D(Pq^{kQ\p_Q},Q,q)/k(1-q^k\big)}.
\]
The orbit of $(1-q) J^T_{\tilde{X}}$ under this group is the whole of $\L_{\tilde{X}}$ $\big($and moreover, picking suitable operators as described in~\cite{Givental:perm8} one obtains an explicit parameterization of $\L_{\tilde{X}}\big)$.

\begin{Theorem}\label{thm3}
Elements of the orbit of $(1-q) J^T_{\Pi \mathfrak{g}/\mathfrak{t}}$ under the subgroup ${\mathcal P}^W$ of the operators invariant with respect to the Weyl group, in the specialization $Q_1=\cdots=Q_n=Q$, $\Lambda_0=1$ turn into values of the big J-function of the grassmannian.
\end{Theorem}

We conjecture that a similar result holds universally for non-singular GIT quotients, i.e., that $\L_{C//G}$ is obtained from $\L_{\Pi \mathfrak{g}/\mathfrak{t}}^W$ (where $\Pi \mathfrak{g}/\mathfrak{t}$ is the super-space over the base $C//T$ defined as explained above) by specializing the Novikov ring to its $W$-invariant part, and passing to the limit $\Lambda_0=1$.

\begin{proof}
The proof of the theorem is based on $T^N$-fixed point localization. It should be obvious after Section \ref{Sec3} that the criterion (i) of the fixed point method is invariant under ${\mathcal P}$ even before the specialization to $Q_i=Q, \Lambda_0=1$. To verify (ii), take localizations $\big(J^T_{\Pi \mathfrak{g}/\mathfrak{t}}\big)_{(1,\dots,n)}$ and $\big(J^T_{\Pi \mathfrak{g}/\mathfrak{t}}\big)_{(n+1, 2,\dots,n)}$ of $J^T_{\Pi \mathfrak{g}/\mathfrak{t}}$ (thinking of the fixed points $\mathop{\rm Span}(e_1,\dots,e_n)$ and $\mathop{\rm Span}(e_2,\dots,e_{n+1})$ in the grassmannian). The ambiguity in the ordering of the values $P_i=\Lambda_{i'}$ becomes irrelevant in the limit $Q_1=\dots=Q_n=Q$ due to the $W$-symmetry. Before the limit, we take here $P_1=\Lambda_1$,~\dots, $P_n=\Lambda_n$ for the first fixed point, and $P_1=\Lambda_{n+1}$, $P_2=\Lambda_2$,~\dots, $P_n=\Lambda_n$ for the second. For $x=(\Lambda_1/\Lambda_{n+1})^{-1/m_0}$, we have
\[
\Res_{q=x} \big(J^T_{\Pi \mathfrak{g}/\mathfrak{t}}\big)_{(1,\dots,n)}\frac{{\rm d}q}{q} = -\frac{Q_1^{m_0}}{m_0} {\rm Coeff}_{(1,\dots,n)}^{(2,\dots,n+1)}(m_0)\big(J^T_{\Pi \mathfrak{g}/\mathfrak{t}}\big)_{(n+1, 2,\dots,n)}\big|_{q=x}.
\]
This is simply the recursion relation (ii) for the target space $\tilde{X}=\big(\CC P^{N-1}\big)^n$ corresponding to the 1-dimensional $T^N$-orbit connecting two fixed points, $\mathop{\rm Span}(e_1)$ and $\mathop{\rm Span}(e_{n+1})$ in projection to the first factor $\CC P^{N-1}$, and constant (and equal to $\mathop{\rm Span}(e_i)$, $i=2,\dots,n$) in the other projections. The recursion coefficient here turns into the correct one for the grassmannian in~the limit $\Lambda_0=1$ and $Q_1=Q$.

Note that operators $P_i^kq^{kQ_i\p_{Q_i}}$ specialize to $\Lambda_i^k q^{kQ_i\p_{Q_i}}$ at the fixed point $\mathop{\rm Span}(e_1,\dots,e_n)$, and for $i>1$ commute with $Q_1^{m_0}$, while for $i=1$ we have
\[
\Lambda_1^k q^{kQ_1\p_{Q_1}} Q_1^{m_0}=Q_1^{m_0}q^{km_0} \Lambda_1^k q^{kQ_1\p_{Q_1}} \equiv_{\mod 1-q^{m_0}\frac{\Lambda_1}{\Lambda_{n+1}}}
Q_1^{m_0}\Lambda_{n+1}^{k} q^{kQ_1\p_{Q_1}}.
\]
This implies that for any finite difference operator $D$ regular at $q=(\Lambda_1/\Lambda_{n+1})^{-1/m_0}$, the localizations $\big(D J^T_{\Pi \mathfrak{g}/\mathfrak{t}}\big)_{(1,\dots,n)}$ and $\big(D J^T_{\Pi \mathfrak{g}/\mathfrak{t}}\big)_{(n+1,2,\dots,n)}$ of $D J^T_{\Pi \mathfrak{g}/\mathfrak{t}}$ also satisfy the above recursion relation with the same recursion coefficient. In fact this direct verification is not even necessary, since it simply elucidates in terms of fixed point localization the general fact that $\L_{\Pi \mathfrak{g}/\mathfrak{t}}$ is ${\mathcal P}$-invariant.

We conclude that when $D$ is $W$-invariant, $(1-q) D J^T_{\Pi \mathfrak{g}/\mathfrak{t}}$ specializes at $Q_1=\dots=Q_n=Q$ and $\Lambda_0=1$ into a point in the loop space $\K$ (corresponding to the grassmannian) which satisfies the correct recursion relation, and hence belongs to $\L_{{\rm Gr}_{n,N}}$.
\end{proof}

\begin{Remark}\label{rmk4}
Of course, the above argument applies more generally than the grassmannian example, and works whenever a torus ($T^N$ in this case) acts on $C//T^n$ and $C//G$ with isolated fixed points and isolated one-dimensional orbits. In particular, it applies to {\em twisted} quantum K-theories studied in~\cite{Givental:perm11} and generalizing the above transition from $\tilde{X}=\big(\CC P^{N-1}\big)^n$ to $\Pi \mathfrak{g}/\mathfrak{t}$. Namely, let $E=E(P_1,\dots,P_n) \in K^0_{T^N}({\rm Gr}_{n,N})$ be a virtual vector bundle (for this, $E$ needs to~be symmetric in $P_i$). It can be used to ``twist'' the virtual structure sheaves of the moduli spaces of stable maps~-- for both targets, $\Pi \mathfrak{g}/\mathfrak{t}$ and ${\rm Gr}_{n,N}$:
\[
\O_{g,m,d} \mapsto \O_{g,m,d}\otimes {\rm e}^{ \sum_{k\neq 0} \Psi^k(\mu_k \ft_*\ev*E)/k},
\]
where $\mu_k$ are some prefixed elements of the ground ring $R$ (and, abstractly speaking, should better be taken from $R_{+}$ as a precaution lest the modifying expression diverges). Then the big J-functions in the twisted quantum K-theories of $\Pi \mathfrak{g}/\mathfrak{t}$ and ${\rm Gr}_{n,N}$ are related the same way as described in the theorem: For any $W$-invariant value of the twisted J-function of $\Pi \mathfrak{g}/\mathfrak{t}$ (in place of $J^T_{\Pi \mathfrak{g}/\mathfrak{t}}$), the elements of its orbit under ${\mathcal P}^W$ in the limit $Q_1=\dots=Q_n=Q$ specialize into values of the big J-function in the twisted quantum K-theory of the grassmannian.
\end{Remark}

\section[Balanced I-functions and T* Gr {n,N}]{Balanced I-functions and ${\boldsymbol T^*{\mathbf{Gr}}_{n,N}}$}\label{Sec5}

In some recent literature motivated by representation theory (see, e.g.,~\cite{KPSZ, Ok, PSZ}), quantum K-theory of {\em symplectic} quiver varieties plays a role, and among them, the cotangent bundles of~the grassmannians (rather than the grassmannians {\em per se}) take the place of the target spaces. K-theoretic computations in the quasimap compactifications of spaces of rational curves in~such targets lead A.~Okounkov and his followers to $q$-hypergeometric functions quite interesting from the point of view of the theory of integrable systems. To illustrate one specific property (appa\-rently important in their theory for technical reasons) consider the series
\begin{gather*}
I^T= \sum_{0\leq d_1,\dots,d_n}Q^{d_1+\cdots+d_n} \prod_{i=1}^n\prod_{j=1}^N\prod_{m=1}^{d_i}\frac{1-q^mYP_i/\Lambda_j} {1-q^mP_i/\Lambda_j}
\\ \hphantom{I^T=\sum_{0\leq d_1,\dots,d_n}Q^{d_1+\cdots} }
{}\times \prod_{i,j=1}^n\frac{\prod_{m=-\infty}^{d_i-d_j}(1-q^mP_i/P_j)}{\prod_{m=-\infty}^0(1-q^mP_i/P_j)}
\frac{\prod_{m=-\infty}^0(1-q^mY P_i/P_j)}{\prod_{m=-\infty}^{d_i-d_j}(1-q^mY P_i/P_j)}.
\end{gather*}
Here $Y\in \CC^{\times}$ (denoted in~\cite{Ok} and elsewhere by $\hbar$) represents the circle acting by scalar multiplication on the fibers of a vector bundle over the compact base (which is meant to be $T^*{\rm Gr}_{n,N}$ in our example). Note that the series is formed of fractions $(1-q^mY X)/(1-q^m X)$, which are bounded both as $q\to 0$ and $q\to \infty$. In the fixed-point computations on quasimap spaces of~symplectic targets, this property of generating functions being {\em balanced} (in terminology of~\cite{Liu}) is a by-product of tensoring the virtual structure sheaf with the square root of the determinant bundle of the moduli space (i.e., in effect computing indexes of real Dirac operators rather than holomorphic Euler characteristics). The questions we will address here are about the place of~the
series $I^T$ and its close counterparts in the ``genuine'' (i.e., based on stable map compactifications) quantum K-theory of the grassmannian: Does $I^T$ represent a value of the big J-function of {\em any} version of quantum K-theory, and if so, then what version and on which space? Is it the small J-function in that theory? We will give several different affirmative answers to the first question, and negative to the second.

\begin{Theorem}\label{thm4}
The series $(1-q) I^T$ represents a value of the torus-equivariant, permutation-invariant big J-function of $\Pi T {\rm Gr}_{n,N}$ (the odd tangent bundle of the grassmannian).
\end{Theorem}
Initially the interest in GW-theory of $\Pi E$ for a bundle $E$ over a compact base is motivated by the fact that in the non-equivariant limit $Y\to 1$, GW-invariants of $\Pi E$, when the limit exist, turn into GW-invariants of the zero locus of a generic section of $E$ (which in the case of $E=TX$ consists of $\chi (X)$ isolated points).

\begin{proof}
We can follow the same route as that of Theorem~\ref{thm2}. The localization
\begin{gather*}
I^T_{(1,\dots,n)}= \sum_{0\leq d_1,\dots,d_n}Q^{d_1+\cdots+d_n} \prod_{i=1}^n\prod_{j=1}^N\prod_{m=1}^{d_i}\frac{1-q^mY\Lambda_i/\Lambda_j} {1-q^m\Lambda_i/\Lambda_j} \\ \phantom{I^T_{(1,\dots,n)}= \sum_{0\leq d_1,\dots,d_n}Q^{d_1+\cdots}}
{}\times \prod_{i,j=1}^n\frac{\prod_{m=-\infty}^{d_i-d_j}(1-q^m\Lambda_i/\Lambda_j)} {\prod_{m=-\infty}^0(1-q^m\Lambda_i/\Lambda_j)}\frac{\prod_{m=-\infty}^0(1-q^mY \Lambda_i/\Lambda_j)} {\prod_{m=-\infty}^{d_i-d_j}(1-q^mY \Lambda_i/\Lambda_j)}
\end{gather*}
of $I^T$ at the fixed point $(1,\dots,n)$ in the grassmannian, together with such localizations at other fixed points, pass the test (i) of the fixed point theory, and the residues at the poles $q=(\Lambda_j/\Lambda_i)^{1/m_0}$ satisfy the recursion relation of the familiar form (ii) with suitable recursion coefficients. This should be obvious after our analysis of the series $J^T$ in Sections~\ref{Sec3} and~\ref{Sec2} respectively. Moreover, according to Remark~\ref{rmk3}, it only suffices to match the values of these recursion coefficients with those in GW-theory of $\Pi T {\rm Gr}_{n,N}$. Using the notation $x=(\Lambda_{n+1}/\Lambda_1)^{1/m_0}$ and our result from Section~\ref{Sec2}, we find the recursion coefficient corresponding to the pole at $q=x$ in~the form
\[
-\frac{Q^{m_0}}{m_0} \frac{\Eu(T_{(1,\dots,n)}X)}{\Eu(T_{\phi}X_{0,2,m_0})}
\]
as before, times the modifying factor
\[
\prod_{j=1}^N\prod_{m=1}^{m_0}(1-x^mY\Lambda_1/\Lambda_j)\ \prod_{i=2}^n\prod_{m=1}^{m_0} \frac{1-x^mY\Lambda_i/\Lambda_{n+1}}{1-x^mY\Lambda_1/\Lambda_i},
\]
where the target $X={\rm Gr}_{n,N}$. Unsurprisingly, the modifying factor is almost reciprocal to the expression for $\Eu(T_{(1,\dots,n)}X)/\Eu(T_{\phi}X_{0,2,m_0})$. They differ by the presence of $Y$ in each factor, and by the extra factor $1-x^{m_0}Y\Lambda_1/\Lambda_{n+1}$ (actually equal to $1-Y$, and excluded from the expression in Section~\ref{Sec2} where $Y=1$). In our computation of $H^0\big(\CC P^1;\phi^*(TX)\big)$, the latter (zero) factor represents the line spanned by the vector field $z^{m_0}\p_{z^{m_0}}$ (infinitesimally rescaling the target~$\CC P^1$), and falls out of $T_{\phi}X_{0,2,m_0}$ because of the infinitesimal automorphism $z\p_z$ of the source $\CC P^1$. Thus, the factor $1-Y$ remains present in $\Eu_{\CC^{\times}}(\ft_*\phi^*(TX))$. Note that $Y$ was introduced as the character of $\CC^{\times}$-action on $T^*X$. The action on $TX$ is given therefore by $Y^{-1}$, but the definition of the K-theoretic Euler class as the exterior algebra of the {\em dual} bundle restores the factors~$Y$ everywhere. Thus, the modifying factor coincides with $\Eu_{\CC^{\times}}(\ft_*\phi^*(TX))/\Eu_{\CC^{\times}}(T_{(1,\dots,n)}X)$, and the recursion coefficient altogether has the required form
\[
-\frac{Q^{m_0}}{m_0} \frac{\Eu(T_{(1,\dots,n)}\Pi TX)}{\Eu (T_{(1,\dots,n)}(\Pi TX)_{0,2,m_0})}.\tag*{\qed}
\]
\renewcommand{\qed}{}
\end{proof}

\begin{Corollary}\label{cor1}
The series $(1-q) I^T /(1-Yq)$ represents a value of the torus-equivariant, per\-mu\-ta\-tion-invariant big J-function in the quantum Hirzebruch K-theory of the grassman\-nian~${\rm Gr}_{n,N}$.
\end{Corollary}

Recall that the Hirzebruch $\chi_{-Y}$-genus of a compact complex manifold $M$ is defined by
\[
\chi_{-Y}(M):=\sum_{p=0}^{\dim M} (-Y)^p \chi (M; \Omega^p(M)) = H^*(M; \Eu_{\CC^{\times}}(TM)),
\]
where the rightmost interpretation assumes that $Y\in \CC^{\times}$ acts fiberwise on the tangent bundle by $Y^{-1}$. More generally, one can define the (classical) {\em Hirzebruch K-theory} by replacing the structure sheaves $\O_M$ with $\O_M\otimes \Eu_{\CC^{\times}}(TM)$. The {\em quantum} Hirzebruch K-theory of a target $X$ is defined by similarly modifying the virtual structure sheaves of the moduli spaces $X_{g,m,d}$ using their virtual tangent bundles:
\[
\O_{g,m,d} \mapsto \O_{g,m,d} \otimes \Eu_{\CC^{\times}}(T X_{g,m,d}).
\]
According to a result from~\cite{G-H}, the theory thus obtained can be expressed via the ordinary quantum K-theory, implying in particular Corollary~\ref{cor1} (see Remark~\ref{rmk5} below). However, it~also follows from our fixed point approach. Namely, the big J-function of (permutation-invariant) quantum Hirzebruch K-theory has the form
\[
\frac{1-q}{1-qY}+\t(q)+\sum_{d,m,\alpha}Q^d\phi_{\alpha}\bigg\lan \phi^{\alpha} \frac{1-qYL_0}{1-qL_0},\t(L_1),\dots, \t(L_m)\bigg\ran^{S_m}_{0,m+1,d},
\]
where the correlators are defined using the virtual structure sheaves of the Hirzebruch K-theory. This is not an {\em ad hoc} definition, but is dictated by the general formalism; the dilaton shift and the first input embody respectively: the Euler class (of the universal line bundle $q^{-1}$) corresponding to the genus, and the reciprocal of the equivariant Euler class of $L_0^{-1}$. Consequently, the recursion coefficient of the fixed point theory acquires a new factor $1-Y$: the residue of $\frac{1-qYL_0}{1-qL_0} \frac{{\rm d}q}{q}$ at the pole $q=L_0^{-1}$ (equal in our computations to $(\Lambda_1/\Lambda_{n+1})^{-1/m_0}$). But the above explanation why this factor belongs to $\Eu_{\CC^{\times}}(\ft_*\phi^*(TX))$ means it {\em does not} belong to $\Eu_{\CC^{\times}}(T_{\phi}X_{0,2,m_0})$. The~latter occurs in Lefschetz' fixed point formula for the modified virtual structure sheaf. The net result is that the recursion relation (ii) remains the same as in the theory of $\Pi T X$. Note that a~scalar factor, such as $1/1-qY$ in $(1-q) I^T/(1-qY)$ has no effect on the recursion relation (a~fact indicating that the range $\L\subset \K$ of the big J-function is an ``overruled cone''). The role of this factor is to guarantee that modulo~$Q$, the series equals the dilaton shift, and hence the rest of the series is $Q$-adically small as required.

\begin{Remark}\label{rmk5}
By the way, $1/(1-qY)=\sum_{m\geq 0} Y^mq^m$ is considered a ``Laurent polynomial'' in~$q$, i.e., an element of~$\K_{+}$ in Hirzebruch K-theory, as it doesn't have poles relevant in localization theory. The correlator part of the big J-function in the quantum Hirzebruch K-theory clearly satisfies $\J|_{q=\infty}=Y\J|_{q=0}$, and this condition defines the new space $\K_{-}$. The general result of~\cite{G-H}, which applies to the all-genera permutation-equivariant quantum K-theory, says that the total descendant potential $\D_X^Y$ for the Hirzebruch version of the theory is obtained from the ``ordinary'' one, $\D_X^0$, by three transformations: the Eulerian twisting corresponding to the bundle $E=TX-1$ (in genus $0$, this has practically the same effect as the twisting by $TX$, producing the big J-function of $\Pi T X$), and the above changes in the dilaton shift and polarization \mbox{$\K=\K_{+}\oplus \K_{-}$}. The transformations correspond to the three summands in the virtual tangent bundles:
\[
TX_{g,m,d}=\ft_*\ev^*(TX-1)+\ft_*\big(1-L^{-1}\big)-(\ft_*j_*\O_Z)^*,
\]
where $L$ is the universal cotangent line over $X_{g,m+1,d}$ at the $m+1$-st marked point, and $j\colon Z\to X_{g,m+1,d}$ is the inclusion of the nodal locus. Here $\ft_*\ev^*TX$ represents variations of stable maps from pointed curves with a fixed complex structure, $\ft_*(L^{-1})$ represents variations of the complex structure of the curves, while the last term is supported on the virtual divisor $\ft(Z)\subset X_{g,m,d}$ where the combinatorics of the curves changes, and accounts for the difference between the virtual tangent bundle and the sheaf of vector fields tangent to this divisor.
\end{Remark}

Another form of Theorem~\ref{thm4} can be derived from Serre's duality. The cotangent line bundle~$L$ of a pointed nodal curve and its canonical bundle $K$ are related by $L=K(D)$, where $D:=\sum_{i=1}^m\sigma_i$ is the divisor of the marked points (i.e., away from the nodes, a section of $L$ is a~differential allowed to have 1st order poles at the markings).

Given a bundle $E$ over $X$, on $X_{g,m,d}$ we have
\[
\ft_*\ev^*E=-\big(\ft_* K\ev^*E^{\vee}\big)^{\vee} = -\big(\ft_*\ev^*E^{\vee}\big)^{\vee}+\big(\ft_*(1-L)\ev^*E^{\vee}\big)^{\vee}-\sum_{i=1}^m\ev_i^*E.
\]
Applying the quantum Adams--Riemann--Roch (Theorem 2 in~\cite{Givental:perm11}), we find that tensoring of~$\O_{g,m,d}$ with $\Eu^{-1}_{\CC^{\times}}\big(\big(\ft_*(1-L)\ev^*E^{\vee}\big)^{\vee}\big)$ in the correlators of permutation-equivariant quantum K-theory is equivalent to the change $(1-q) \mapsto (1-q)\Eu_{\CC^{\times}}^{-1}(E)$ in the dilaton shift. The~same change of this inputs: $\t \mapsto \Eu^{-1}_{\CC^{\times}}(E)\t$, is effected by tensoring with the Euler classes of~$-\ev_i^*E$. In other words, the dilaton-shifted total descendant potential of the theory twisted by $\Eu^{-1}\big(\big(\ft_*\ev^*E^{\vee}\big)^{\vee}\big)$ is obtained from the one twisted by $\Eu_{\CC^{\times}}(\ft_*\ev^*E)$ by the transformation $\D_{\Pi E}(\q) \mapsto \D_{\Pi E} \big(\Eu^{-1}_{\CC^{\times}}(E)\q\big)$. The potentials are considered as quantum states in suitable Fock spaces, and the transformation is induced by the map $\f \mapsto \Eu^{-1}_{\CC^{\times}}(E)\f$ between two copies of the loop space~$\K$ (equipped with two different symplectic structures: based on the Poincar\'e pairing $\chi \big(X; \Eu_{\CC^{\times}}^{-1}(E) ab\big)$ on the source space, and $\chi (X;\Eu_{\CC^{\times}}ab)$ on the target. Consequently, the big \mbox{J-functions} in the genus-$0$ theory are related by the inverse transformation: $\J_{\Pi E}\mapsto \Eu_{\CC^{\times}}(E)\J_{\Pi E}$. Applying all this to $E=T{\rm Gr}_{n,N}$, we arrive at the following conclusion.

\begin{Corollary}[{cf.~\cite{Liu}}] \label{cor2}
The series $\Eu_{\CC^{\times}}(T{\rm Gr}_{n,N})(1-q)I^T$ represents a value of the big J-func\-tion in the torus-equivariant, permutation-invariant quantum K-theory of the grassmannian~${\rm Gr}_{n,N}$ twisted by
\[
\O_{0,m,d}\mapsto \O_{0,m,d}\otimes \Eu^{-1}_{\CC^{\times}}\big((\ft_*\ev^*(T^*{\rm Gr}_{n,N}))^{\vee}\big).
\]
\end{Corollary}

One more way of modifying virtual structure sheaves, which was recently introduced and explored by Y.~Ruan and M.~Zhang~\cite{R-Z}, consists in tensoring $\O_{g,m,d}$ with a power of the determinant line bundle $(\det (\ft_*\ev^*E))^{-l}$, thereby bringing the {\em level structure} (of level $l$) into the quantum K-theory. Note that in terms of K-theoretic Chern roots $L_1,\dots, L_M$ of a vector bundle $\E$,
\[
\frac{\Eu (\E)}{\Eu (\E^*)}=\frac{\prod_{k=1}^M \big(1-L_k^{-1}\big)}{\prod_{k=1}^M(1-L_k)}
=(-1)^M L_1^{-1}\cdots L_M^{-1} =(-1)^{\dim \E} (\det \E)^{-1}.
\]
So, we take $E=T^*{\rm Gr}_{n,N}$, $\E=\ft_*\ev^*(T^*{\rm Gr}_{n,N})$, and describe the modification of $\O_{g,m,d}$ used in~Corollary~\ref{cor2} as tensoring with both $\Eu^{-1}_{\CC^{\times}}(\E)$ and $(\det \E)^{-1}$.
After the first operation we land in~the theory of the noncompact bundle space $T^*{\rm Gr}_{n,N}$, and after the second in the level~$1$ version of~this theory. The Poincar\'e pairing changes accordingly into $\chi \big(X; \Eu^{-1}_{\CC^{\times}}(T^*{\rm Gr}_{n,N})$ $(\det T^*{\rm Gr}_{n,N})^{-1} ab\big)$. By the Riemann--Roch formula, $\dim \E = (1-g)\dim X+\int_X c_1(T^*X)$, which for $g=0$ yields $(-1)^{\dim \E}=(-1)^{\dim {\rm Gr}_{n,N}}(-1)^{Nd}$. The first sign is absorbed by the ratio of the Euler classes (of~$T$ and~$T^*$) in the Poincar\'e pairings, and the second by the change $Q\mapsto (-1)^N Q$, leading to the following conclusion.

\begin{Corollary}\label{cor3}
The series
\[
\det (T^*{\rm Gr}_{n,N})^{-1}\Eu_{\CC^{\times}}(T^*{\rm Gr}_{n,N})\cdot (1-q) I^T\big((-1)^N Q\big)
\]
represents a value of the big J-function in the level~$1$, torus-equivariant, permutation-invariant quantum K-theory of the cotangent bundle space $T^*{\rm Gr}_{n,N}$.
\end{Corollary}

Explicitly, the product of the determinant and the Eulerian pre-factor differs by the sign $(-1)^{\dim {\rm Gr}_{n,N}}$ from
\[
\Eu_{\CC^{\times}}(T{\rm Gr}_{n,N})=\frac{\prod_{j=1}^N\prod_{i=1}^n(1-YP_i/\Lambda_j)}{\prod_{i,j=1}^n(1-Y P_i/P_j)}.
\]
Because of this pre-factor, the series even modulo $Q$ is not equal the dilaton shift $1-q$ (as well as in Corollary~\ref{cor2}), which already disqualifies it for the role of the ``small'' J-function.

In fact the $q$-hypergeometric series which arises in the K-theoretic computations on the spaces of quasimaps to the grassmannian is slightly different from the one in Corollaries~\ref{cor2} and~\ref{cor3}. It~has the form
\begin{gather*}
\tilde{I}^T = \sum_{0\leq d_1,\dots,d_n}Q^{d_1+\cdots+d_n} \prod_{i=1}^n\prod_{j=1}^N\prod_{m=0}^{d_i-1}\frac{1-q^mYP_i/\Lambda_j} {1-q^mP_i/\Lambda_j}
\\ \phantom{\tilde{I}^T =\sum_{0\leq d_1,\dots,d_n}Q^{d_1+\cdots}}
{}\times \prod_{i,j=1}^n\frac{\prod_{m=-\infty}^{d_i-d_j}(1-q^mP_i/P_j)}{\prod_{m=-\infty}^0(1-q^mP_i/P_j)}\frac{\prod_{m=-\infty}^{-1}(1-q^mY P_i/P_j)}{\prod_{m=-\infty}^{d_i-d_j-1}(1-q^mY P_i/P_j)},
\end{gather*}
which differs from $I^T$ in that in the products of the factors $1-q^mYP_i/\Lambda_j$, the range of $m$ is not from $1$ to $d_i$ (as for the factors without $Y$) but from $0$ to $d_i-1$ (and similarly for the factors $1-q^mYP_i/P_j$). We claim, however, that $(1-q) \tilde{I}$ and $(1-q) (-1)^{\dim {\rm Gr}_{n,N}} \tilde{I}^T\big((-1)^NQ\big)$ represent some values of the big J-functions of the same theories as described in Corollary~\ref{cor2} and~\ref{cor3} respectively.

Namely, consider the version of $\tilde{I}^T$ with $Q^{d_1+\cdots+d_n}$ is replaced with $Q_1^{d_1}\cdots Q_n^{d_n}$, and apply to it the operator
\[
\frac{\prod_{j=1}^N\prod_{i=1}^n\big(1-Y q^{Q_i\p_{Q_i}}P_i/\Lambda_j\big)}{\prod_{i,j=1}^n\big(1-Y q^{Q_i\p_{Q_i}-Q_j\p_{Q_j}}P_i/P_j\big)}.
\]
This results in restoring the ``missing'' factors with $m=d_i$ or $m=d_i-d_j$, and in the limit $Q_1=\cdots=Q_n=Q$ yields the series of Corollary~\ref{cor3} (modulo to the sign $(-1)^{\dim_{{\rm Gr}_{n,N}}}$ and the change $Q\mapsto (-1)^NQ$). On the other hand, the operator can be written as
\[
{\rm e}^{-\sum_{k>0} Y^k \left[\sum_{j=1}^N\sum_{i=1}^n q^{kQ_i\p_{Q_i}}P^k_i/\Lambda^k_j - \sum_{i,j=1}^n q^{kQ_i\p_{Q_i}-kQ_j\p_{Q_j}} P^k_i/P^k_j \right]/k},
\]
and hence belongs to the group ${\mathcal P}^W$, which justifies our claim due to Theorem~\ref{thm3} and Remark~\ref{rmk4}.

The series $(1-q)\tilde{I}^T$ appears to have better chances to pose for the ``small'' J-function, because the term with $Q^0$ is $1-q$, and other terms are {\em reduced} rational functions of~$q$. And indeed, H.~Liu~\cite{Liu}, looking for a stable-map K-theory interpretation of the $q$-hypergeometric series arising in the quasimap K-theory of quiver varieties, shows that in the case $n=1$ of projective spaces, the series $(1-q)\tilde{I}^T$ is the small J-function in the theory described by Corollary~\ref{cor2}. However, he falls short of sticking to this interpretation, because he finds an example (namely the manifold of flags in $\CC^3$) where the similarly twisted small J-function is unbalanced.

In fact none of $(1-q) \tilde{I}^T$ with $n>1$ (and none of other I-series featuring in this section) represent ``small'' J-functions, and not because some rational functions are not reduced, but for much more dramatic reasons. Namely, in our fixed point characterization of the big J-function, the poles participating in the recursion relations come from the characters of the torus action on the grassmannian {\em per se}: $q=(\Lambda_i/\Lambda_j)^{-1/m}$. The terms of a balanced I-series containing the factors $1/(1-q^mYP_i/P_j)$ lead to the poles at $q=(Y\Lambda_i/\Lambda_j)^{-1/m}$, which cannot come from fixed point localization. Such fractions should therefore be interpreted as elements of $\K_{+}$, i.e., as geometric series $\sum_{k\geq 0}q^{mk}(YP_i/P_j)^k$ converging in the $Y$-adic topology. Thus, representing $(1-q) \tilde{I}^T$ as $(1-q)+\t(q) \ \mod \K_{-}$ results in a very complicated value of $\t(q)$, meaning that the series represents the value of the big J-function with the inputs $\t(L_i)$ which are rather far from~$0$. What makes the effect even more dramatic is that it is the input in the {\em permutation-invariant} quantum K-theory, no counterpart of which has been discussed so far in the context of quasimap spaces.

Apart from this, the interpretation of the series given in Corollary~\ref{cor3} is quite parallel to~its definition~\cite{KPSZ, Ok, PSZ} in the quasimap theory as a generating function capturing some K-theoretic GW-invariants of the cotangent bundle of the grassmannian based on the virtual structure sheaves ``symmetrized'' by the determinant factors.

Finally, we would like to stress that, although we have formulated Theorem~\ref{thm4} and its corollaries as statements about the particular I-function, modified slightly in one way or another, in fact these modifications affect the recursion coefficients in a simple and controllable way, implying that the whole big J-functions of the respective theories coincide up to these minor modifications. In particular, Corollary~\ref{cor3} is connected to Theorem~\ref{thm4} by a general phenomenon called the ``non-linear (or quantum) Serre duality''~\cite{Coates-Givental}. It relates GW-invariants of the super-space $\Pi E$ and bundle space $E^{\vee}$, and was first observed in~\cite{Givental:EGWI} (for cohomological GW-invariants) via fixed point localization. For the full treatment (including higher genus) of the K-theoretic reincarnation of the quantum Serre duality we refer to~\cite{Yan}.

\section{Non-abelian quantum Lefschetz}\label{Sec6}

A somewhat different proof of Theorem~\ref{thm4} could be derived from Theorem~\ref{thm3} together with Remark~\ref{rmk4}, applied to GW-invariants of the grassmannian Euler-twisted by the tangent bundle
\begin{gather*}
E=\sum_{j=1}^N\sum_{i=1}^n\Lambda_j/P_i-\sum_{i,j=1}^nP_j/P_i.
\end{gather*}
Here we illustrate this approach using as an example Eulerian twistings applied to the {\em dual} tautological bundle $E=P_1^{-1}+\cdots+P_n^{-1}$.

The Euler-twisted theory (of both $\Pi \mathfrak{g}/\mathfrak{t}$ and ${\rm Gr}_{n,N}$) is defined by
\[
\O_{g,m,d}\mapsto \O_{g,m,d}\otimes \Eu_{\CC^{\times}}(\ft_*\ev^*E) = \O_{g,m,d}\otimes {\rm e}^{-\sum_{k>0} Y^k \Psi^{-k}(\ft_*\ev^* E)/k},
\]
where $Y\in \CC^{\times}$ acts by multiplication on the fibers of $E$. According to the quantum Adams--Riemann--Roch theorem~\cite{Givental:perm11}, the twisted theory is obtained from the untwisted one by the multiplication: $\L_{\Pi E}=\square^{-1} \L_{{\rm Gr}_{n,N}}$, where
\[
\square:= {\rm e}^{ \sum_{k>0} Y^k\Psi^{-k}(E) q^k/k(1-q^k)}
= \Eu_{\CC^{\times}}(E) \sum_{d_1,\dots,d_n\geq 0}\frac{Y^{d_1+\cdots+d_n}P_1^{d_1}\cdots P_n^{d_n}}{\prod_{i=1}^n\prod_{m=1}^{d_i}(1-q^mP_i)}.
\]

This is a convenient moment to address one general technical issue. Values of big J-functions are supposed to lie in $R_{+}$-neighborhood of the dilaton shift $1-q$. The terms containing Novikov's variables are $R_{+}$-small, and remain such after multiplication by anything like $\square$. Moreover, for Laurent polynomials $\t(q)$ with $R_{+}$-small coefficients, $\square\ \t$ contains only finitely many non-reduced terms, and so modulo $\K_{-}$ it remains a Laurent polynomial (with $R_{+}$-small coefficients). However, the product $\square (1-q) \equiv (1-q) + Y E^{\vee} \mod \K_{-}$ seems to present a problem. One way to resolve it is to postulate that $R_{+}\ni Y$. Here is a better way to deal with this issue, which is especially useful if one also needs to use $Y^{-1}$ in the same context. Consider the operator
\[
D={\rm e}^{ \sum_{k>0} Y^k q^k \sum_{i=1}^n P_i^k(1-q^{k Q_i\p_{Q_i}})/k(1-q^k)},
\]
which is $\square$ times the pseudo-finite-difference operator from the group ${\mathcal P}^W$ corresponding to the finite-difference operator $-qY\sum_{i=1}^n P_i q^{Q_i\p_{Q_i}}$. Therefore
\[
\square^{-1} \L_{\Pi \mathfrak{g}/\mathfrak{t}} = D^{-1}\L_{\Pi \mathfrak{g}/\mathfrak{t}},
\]
which by Theorem~\ref{thm3} (or rather its generalization explained in Remark~\ref{rmk4}) turns into $\L_{\Pi E}$ in~the limit $Q_1=\cdots=Q_n=Q$, $\Lambda_0=1$. The advantage of using $D$ instead of $\square$ is that $D$ does not change terms constant in $Q_1,\dots,Q_n$: $D(1-q)=1-q$.

\begin{Theorem}[{non-abelian quantum Lefschetz}]\label{thm5}
Suppose
\[
\sum_{d_1,\dots,d_n\geq 0} I_{d_1,\dots,d_n}(P_1,\dots,P_n,\Lambda_0) Q_1^{d_1}\cdots Q_n^{d_n}
\]
is a $W$-invariant point in $\L_{\Pi \mathfrak{g}/\mathfrak{t}}$, then
\[
\sum_{d_1,\dots,d_n\geq 0}I_{d_1,\dots,d_n}|_{\Lambda_0=1} Q^{d_1+\cdots+d_n}\prod_{i=1}^n\prod_{m=1}^d(1-q^mYP_i)
\]
is a point in $\L_{\Pi E}$.
\end{Theorem}
\begin{proof}Apply $D^{-1}$ and use
\begin{gather*}
{\rm e}^{-\sum_{k>0} q^kY^kP_i^k(1-q^{Q_i\p_{Q_i}})/k(1-q^k)} Q_i^{d_i} = {\rm e}^{-\sum_{k>0}q^kY^kP_i^k(1-q^{d_ik})/k(1-q^k)} Q_i^{d_i}
\\ \hphantom{{\rm e}^{-\sum_{k>0} q^kY^kP_i^k(1-q^{Q_i\p_{Q_i}})/k(1-q^k)} Q_i^{d_i}}
{}= {\rm e}^{ -\sum_{k>0} \sum_{m=1}^{d_i} q^{mk}Y^kP^k/k} Q_i^{d_i}\! = \! Q^d\!\!\prod_{m=1}^{d_i}(1-q^mYP_i). \!\!\! \!\!\tag*{\qed}
\end{gather*}
\renewcommand{\qed}{}
\end{proof}

\begin{Corollary}\label{cor4}
The small J-function of $\Pi E$ equals $(1-q)I$, where
\[
I=\sum_{d_1,\dots,d_n\geq 0}\frac{Q^{\sum d_i}\prod_{i=1}^n\prod_{m=1}^{d_i}(1-q^mYP_i)}{\prod_{j=1}^N\prod_{i=1}^n
\prod_{m=1}^{d_i}(1-q^mP_i/\Lambda_j)}\prod_{i,j=1}^n
\frac{\prod_{m=-\infty}^{d_i-d_j}(1-q^mP_i/P_j)}{\prod_{m=-\infty}^0(1-q^mP_i/P_j)}.
\]
\end{Corollary}

\section{Level structures and dual grassmannians}
\label{Sec7}

Of course, the approach illustrated by Theorem~\ref{thm5} applies to any bundle $E$ over the grassmannian, since $E$ can always be written as a symmetric combination of monomials $\prod_i P_i^{l_i}$. We are going to use this together with the observation (see Section~\ref{Sec5}) that $\det^{-1}(-\E) =\Eu (\E)/\Eu (\E^{\vee})$ in order to describe the effect of the level structure on the genus-0 quantum K-theory of the grassmannian. For the sake of illustration, we take $E$ to be the tautological bundle $V=P_1+\cdots+P_n$. With $\E:=\ft_*\ev^*V$, we have level-$l$ twisted structure sheaves
\[
\O_{g,m,d}\otimes \det^{-l}(-\E) = \O_{g,m,d}\otimes \frac{\Eu(l \E)}{\Eu(l \E^{\vee})} = \O_{g,m,d}\otimes {\rm e}^{-l \sum_{k\neq 0} \Psi^k(\E)/k}.
\]
The Adams--Riemann--Roch theorem from~\cite{Givental:perm11} yields the multiplication operator
\[
\square = {\rm e}^{ -l \sum_{k\neq 0}\Psi^k(V)/k(1-q^k)},
\]
and the respective pseudo-finite-difference operator
\[
D={\rm e}^{ -l \sum_{k\neq 0}\sum_{i=1}^n P^k_i(1-q^{kQ_i\p_{Q_i}})/k(1-q^k)}.
\]
Applying it to $Q_i^{d_i}$, we find
\begin{align*}
D Q_i^{d_i} &= Q_i^{d_i}{\rm e}^{ -l\sum_{k>0}\sum_{m=0}^{d_i-1}(P_i^kq^{mk}-P_i^{-k}q^{-mk})/k}\\
 & = Q_i^{d_i}\prod_{m=0}^{d_i-1}\left(\frac{1-P_iq^m}{1-P_i^{-1}q^{-m}}\right)^l = Q_i^{d_i}(-P_i)^{ld_i}q^{l\binom{d_i}{2}}.
\end{align*}
The above calculation is somewhat formal. The initial determinantal twisting of $\O_{g,m,d}$ and the finial modifying factors are well-defined, but in order to justify intermediate steps, one needs add to $R_{+}$ two variables $Y$, $Y'$, and replace $\E$ and $\E^{\vee}$ with $Y\E$ and $Y'\E^{\vee}$. This will lead to the product of fractions $(1-YP_iq^m)/(1-Y'P_i^{-1}q^{-m})$, where one can pass to the limit $Y=Y'=1$, thus obtaining the following result.

\begin{Theorem}[{cf.~\cite{R-Z}}]\label{thm6}
Suppose
\[
\sum_{d_1,\dots,d_n\geq 0}I_{d_1,\dots,d_n}(P_1,\dots,P_n,\Lambda_0) Q_1^{d_1}\cdots Q_n^{d_n}
\]
is a $W$-invariant point in $\L_{\Pi \mathfrak{g}/\mathfrak{t}}$. Then
\[
\sum_{d_1,\dots,d_n\geq 0}I_{d_1,\dots,d_n}|_{\Lambda_0=1}Q^{d_1+\cdots+d_n} \prod_{i=1}^n \Big( P_i^{d_i}q_i^{\binom{d_i}{2}} \Big)^l
\]
is a point in $\L_{{\rm Gr}_{n,N}}^{(V,l)}$.
\end{Theorem}
Here $\L_{{\rm Gr}_{n,N}}^{(V,l)}$ is the range of the big J-function in the level-$l$ permutation-invariant \mbox{genus-0} quantum K-theory of the grassmannian ${\rm Gr}_{n,N}$, where $V$ is its tautological bundle, and the \mbox{level-$l$} twisted structure sheaves are defined as in~\cite{R-Z}: $\O_{g,m,d}\otimes \det^{-l}(\ft_*\ev^*V)$. Note, that the spurious signs $(-1)^{l\dim \E}=(-1)^{l\sum d_i+l\dim V}$ initially introduced in the determinantal twisting disappears from our ultimate formulation. The first part of it can be absorbed by the change $Q\mapsto (-1)^lQ$ of the Novikov variable, which is offset by the signs of $(-P_i)^{ld_i}$ in our computation. The second part, $(-1)^{l\dim V}$, is the discrepancy in Poincar{\'e} pairings (it affects the notion of~dual bases $\{ \phi_{\a}\}$, $\{ \phi^\a \}$ in the definition of J-functions) which correspond to the two twistings of~$\O_{{\rm Gr}_{n,N}}$: by $\Eu(lV)/\Eu(lV^{\vee})=\prod (-P_i)^{-l}$ and $\det^{-l}(V)=\prod P_i^{-l}$.

The theorem is a non-abelian counterpart of the result of Ruan--Zhang for toric manifolds, obtained in~\cite{R-Z} on the basis of adelic characterization. Both can also be derived by fixed point localization.

\begin{Corollary}\label{cor5}
The series $(1-q) I^T_{(V,l)}$, where
\[
I^T_{(V,l)}=\sum_{0\leq d_1,\dots,d_n}\frac{Q^{d_1+\cdots+d_n}\prod_{i=1}^nP_i^{ld_i}q^{l\binom{d_i}{2}}} {\prod_{i=1}^n\prod_{j=1}^N\prod_{m=1}^{d_i}(1-q^mP_i/\Lambda_j)} \prod_{i,j=1}^n\frac{\prod_{m=-\infty}^{d_i-d_j}(1-q^mP_i/P_j)} {\prod_{m=-\infty}^0(1-q^mP_i/P_j)}
\]
represents a point in $\L_{{\rm Gr}_{n,N}}^{(l)}$, and for $-n<l\leq N-n+1$ is the small J-function of the level-$l$ theory.
\end{Corollary}
\begin{proof}
The formula itself is obtained, of course, from the non-abelian representation of the small J-function $(1-q)J^T=(1-q)I^T_{(V,0)}$ by the recipe described in the theorem. For $l\geq 0$, terms of $(1-q)I^T_{(V,l)}$ have no pole at $q=0$, and for $l\leq N-n+1$ can be shown to be reduced rational functions of $q$ (except the $Q^0$-term $1-q$). Indeed, the difference between the $q$-degrees of the denominator and numerator of the coefficient of $I^T_{(V,l)}$ indexed by $(d_1,\dots, d_n)$ is
\[
N\sum_i\binom{d_i+1}{2}-\sum_{d_i>d_j} \binom{d_i-d_j+1}{2}-l\sum_i\binom{d_i+1}{2}+l\sum_i d_i.
\]
When $l\leq N-n+1$, the binomial sum is non-negative, since for each $i$ the number of $j$ with $d_j<d_i$ does not exceed $n-1$. The linear term is $> 1$ unless all $d_i\neq 1$. Note that in this case the whole expression doesn't depend on $l$, and is still $\geq N-n+1>1$. Thus, even after multiplication by $1-q$ the rational function remains reduced.

For $l<0$, the terms of the series are therefore also reduced, but can have a pole at $q=0$. However, even when this happens, the pole disappears after summing the terms with the same degree $d_1+\cdots+d_n$~-- at least when $l>-n$. This follows from a non-trivial combinatorial result of H. Dong and Y. Wen~\cite{Dong-Wen}, according to which $I^T_{(V,l)}=\tilde{I}^T_{(V^{\vee},-l)}$ for \mbox{$-n<l<N-n$}, where $\tilde{I}^T_{(\tilde{V}^{\vee},-l)}$ is the similar series corresponding to the {\em dual} grassmannian ${\rm Gr}_{N,N-n}:=\mathop{\rm Hom} \big(\CC^{N-n},\CC^{N\vee}\big)//{\rm GL}_{N-n}(\CC)$, and the bundle $\tilde{V}^{\vee}$ {\em dual} to the tautological one:
\[
\sum_{0\leq d_1,\dots,d_n}\frac{Q^{d_1+\cdots+d_{N-n}}\prod_{i=1}^{N-n}\tilde{P}_i^{-ld_i}q^{-l\binom{d_i+1}{2}}}{\prod_{i=1}^{N-n}\prod_{j=1}^N\prod_{m=1}^{d_i}(1-q^m\tilde{P}_i/\Lambda_j^{-1})} \prod_{i,j=1}^{N-n}\frac{\prod_{m=-\infty}^{d_i-d_j}(1-q^m\tilde{P}_i/\tilde{P}_j)}{\prod_{m=-\infty}^0(1-q^m\tilde{P}_i/\tilde{P}_j)}.
\]
Here $\tilde{P}_i$ are K-theoretic Chern roots of the tautological $N-n$-dimensional bundle $\tilde{V}$. Note the characters $\Lambda_j^{-1}$ of the torus $T^N$ action on $\CC^{N\vee}$. Also note the binomial coefficient $\binom{d_i+1}{2}$ $\big($instead of $\binom{d_i}{2}\big)$: this is the effect of using $\tilde{V}^{\vee}$ rather than $\tilde{V}$ in the construction of the determinantal twistings. In the previous section we already had the experience of using $E=V^{\vee}$, from which it is easy to infer the origin of the modifying factors:
\[
\prod_{m=1}^{d_i}\frac{1-q^m\tilde{P}_i}{1-q^{-m}\tilde{P}_i} = \big({-}\tilde{P}\big)^{d_i}q^{\binom{d_i+1}{2}}.
\]
The dual grassmannians are canonically identified by $\big(V\subset \CC^N\big) \mapsto \big(V^{\perp}\subset \CC^{N\vee}\big)$, and the result of~\cite{Dong-Wen} identifies the two expressions as $Q$-series with coefficients in $K^0_T({\rm Gr}_{n,N})=K^0_T({\rm Gr}_{N,N-n})$-valued rational functions of $q$ when $-n<l<N-n$. By the previous estimates of the $q$-degrees $\big($where this time the linear term $l\sum d_i$ isn't present$\big)$, $(1-q)\tilde{I}^T_{(\tilde{V}^{\vee},-l)}$ passes the requirements to be a small J-function when $0\leq -l < N-(N-n)+1$, i.e., $0\geq l\geq -n$. Therefore (though this is not apparent) so does $(1-q)I^T_{V,l}$ at least for $0>l>-n$.
\end{proof}

\begin{example*}\label{ex1}For ${\rm Gr}_{1,N}=\CC P^{N-1}$, we have
\[
(1-q) I^T_{(l)}=(1-q) \sum_{d\geq 0}\frac{Q^d P^l q^{l\binom{d}{2}}}{\prod_{j=1}^N\prod_{m=1}^d(1-q^mP/\Lambda_j)}.
\]
Obviously the series is the small J-function only when $-1< l \leq N$, i.e., the boundaries given by the corollary are sharp.
\end{example*}

\begin{Proposition}$\L_{{\rm Gr}_{n,N}}^{(V,l)}=\L_{{\rm Gr}_{N,N-n}}^{(\tilde{V}^{\vee},-l)}$.
\end{Proposition}

\begin{proof}
From $\tilde{V}^{\vee}=\CC^N/V$ we find
\[
\det (\ft_*\ev^*V) \otimes \det \big(\ft_*\ev^*\tilde{V}^{\vee}\big) = \det \big(\ft_*\ev^*\CC^N\big)
\]
which over moduli spaces of {\em rational} curves equals $\det \CC^N =\prod_{j=1}^N\Lambda_j=\det V \otimes \det \tilde{V}^{\vee}$,
the factor absorbed by the discrepancy in Poincar\'e pairings between the two theories.
\end{proof}

\section{Mirrors}
\label{Sec8}

Consider the improper {\em Jackson integral} (or {\em $q$-integral}), defined as
\[
\int_0^{\infty} f(X)\ X^{-1}{\rm d}_qX :=\sum_{d\in \ZZ} f\big(q^{-d}\big),
\]
in the example
\[
f(X)=X^{\ln \Lambda/\ln q} \prod_{m=1}^{\infty}(1-X/q^m).
\]
For $|q|>1$, the infinite product converges to an entire function of $X$
which coincides with the $q$-exponential function
\[ e_q^{X/(1-q)}=\sum_{d\geq 0}\frac{X^d}{(1-q)\big(1-q^2\big)\cdots\big(1-q^d\big)}.\]
Since the integrand vanishes at $X=q^{-d}$ with $d<0$, the $q$-integral can be computed as
\[
\sum_{d\geq 0}\Lambda^{-d}\prod_{m=d+1}^{\infty}\big(1-q^{-m}\big)=\sum_{d\geq 0}\frac{\Lambda^{-d}\prod_{m=1}^{\infty}\big(1-q^{-m}\big)}{\big(1-q^{-1}\big)\big(1-q^{-2}\big)\cdots \big(1-q^{-d}\big)} =
\frac{\prod_{m=1}^{\infty}\big(1-q^{-m}\big)}{\prod_{m=0}^{\infty}\big(1-q^{-m}/\Lambda\big)}.
\]
The last equality holds for $|\Lambda|>1$, but analytically
extends the value of the $q$-integral to all $\Lambda\neq q^{-m}$, $m=0,1,2,\dots$.
The ratio is closely related to the $q$-gamma function (see~\cite{DK} for a modern treatment of it, including the above $q$-integral representation). In particular, the application of the translation operator $q^{\Lambda\p_{\Lambda}}\colon \Lambda \to q\Lambda$ results in the multiplication of the whole expression by $1-\Lambda^{-1}$. This also follows from the properties of the integrand:
\[
\big(1-q^{X\p_X}/\Lambda\big) f(X) = X f(X) = q^{\Lambda\p_{\Lambda}} f(X),
\]
since the $q$-integral is obviously preserved by the translation $q^{X\p_X}$ of the integrand. The latter property of $q$-integrals will be more useful to us than the previous explicit calculation of their values in terms of $q$-gamma functions.

More generally, one can define improper $q$-integrals using shifted multiplicative
$q$-lattices $\big\{ q^d/A\, |\, d\in\ZZ\big\}$:
\[
\int_0^{\infty/A} g(Y)\ Y^{-1}{\rm d}_qY:=\int_0^{\infty}g(Y/A) Y^{-1}{\rm d}_q Y=\sum_{d\in\ZZ}g\big(q^{-d}/A\big).
\]
We will need such $q$-integrals (still assuming $|q|>1$) for
\[
g(Y)=\frac{Y^{\ln \Lambda/\ln q}}{\prod_{m=0}^{\infty}(1-Y/q^m)}.
\]
The integrand satisfies
\[
\big(1-\Lambda q^{-Y\p_Y}\big) g(Y) = Y g(Y) = q^{\Lambda\p_{\Lambda}} g(Y),
\]
implying that the $q$-integral, if defined, is multiplied by $(1-\Lambda)$ when $\Lambda$
is replaced by $q\Lambda$. In~fact at $A=1$ the $q$-integral is not defined (because of the factor $1-Y$ in the denominator). However, as it is shown in~\cite{DK}
(see formulas~(1.12),~(1.13), and~(1.15) keeping in mind that $q$ there corresponds to
our~$q^{-1}$), it is well-defined for $A\neq q^d$.
The value of this $q$-integral does depend on~$A$, but the properties remain the same.
For the sake of certainty we may use $A=-1$, and indicate this by the notation
$\int_0^{-\infty} g(Y) Y^{-1}{\rm d}_qY$.

Our goal will be to represent the small J-function $J_X^T$ of the grassmannian $X=G_{n,N}$ by suitable Jackson-like integrals in a fashion similar to representing cohomological J-functions by complex oscillating integrals in the mirror theory of, say, toric or flag manifolds. To maintain visual resemblance with complex oscillating integrals, we will denote $X^{-1}{\rm d}_q X$ as ${\rm d}\ln_q X$, and often replace the infinite products in the integrands
with their asymptotical expressions:
\begin{gather*}
\prod_{m=1}^{\infty}(1-X/q^m) \sim {\rm e}^{ \sum_{k>0}X^k/k(1-q^k)},
\qquad
\frac{1}{\prod_{m=0}^{\infty}(1-Y/q^m)} \sim {\rm e}^{ -\sum_{k>0}q^kY^k/k(1-q^k)}.
\end{gather*}

Let us recall from Section~\ref{Sec4} that $J^T_X$ is obtained from $J^T_{\Pi\mathfrak{g}/\mathfrak{t}}$ by passing to the limit $\Lambda_0=1$, $Q_1=\cdots=Q_n=Q$, and takes values in $K^0_T(X)$ consisting of {\em symmetric} functions of $P_1,\dots,P_n$. Before the limit, $J^T_{\Pi\mathfrak{g}/\mathfrak{t}}$ is the J-function of a toric superspace. We begin with setting up the toric mirror (cf.~\cite{Givental:perm6}) to this toric superspace, and studying its properties.

In the complex torus $\X$ with multiplicative coordinates $X_{ij}$, $i=1,\dots,n$, $j=1,\dots,N$, $Y_{ii'}$, $i,i'=1,\dots,n$, $i\neq i'$, consider the n-parametric family of tori
\[
\X_{Q_1,\dots,Q_n}:=\bigg\{ (X,Y)\in \X \bigm| \prod_jX_{ij}=Q_i
\prod_{i'\neq i}(Y_{ii'}/Y_{i'i}),\, i=1,\dots, n \bigg\},
\]
and introduce the $q$-integral
\begin{gather*}
\I:= \int_{\Gamma\subset \X_{Q_1,\dots,Q_n}} {\rm e}^{ \sum_{k>0}\left(\sum_{i,j}X_{ij}^k -q^k\sum_{i\neq i'}Y_{ii'}^{k}\right)/k(1-q^k)}
\\ \hphantom{\I:= \int_{\Gamma\subset \X_{Q_1,\dots}}}
{}\times \prod_{i\neq i'} Y_{ii'} \frac{\bigwedge_i \big(\bigwedge_j {\rm d}_q\ln X_{ij} \bigwedge_{i'\neq i}{\rm d}_q\ln Y_{ii'}\big)}{\bigwedge_i
\big(\sum_j{\rm d}_q\ln X_{ij} -\sum_{i'\neq i}{\rm d}_q\ln (Y_{ii'}/Y_{i'i})\big)}.
\end{gather*}
To clarify the wedge-product expression: if the subscript in all $d_q$ is removed, the expression becomes the standard translation-invariant holomorphic volume on the complex torus $\X_{Q_1,\dots,Q_n}$ (and coincides with the one found in Introduction).

By the ``cycle'' $\Gamma$ we understand a ``multiplicative'' $q$-lattice in $\X_{Q_1,\dots,Q_n}$, i.e., a $\ln q$-lattice on the universal covering of the torus of rank $nN+n^2-2n$, suitable for multi-dimensional $q$-integration, or a formal linear combination of such $q$-lattices; we'll meet some examples later. This should be considered as the K-theoretic mirror to what we denoted in Section~\ref{Sec4} by $\Pi \mathfrak{g}/\mathfrak{t}$: the toric super-bundle over $\big(\CC P^{N-1}\big)^n=\CC^{Nn}//T^n$ with the fiber $\mathfrak{g}/\mathfrak{t}$ associated with the adjoint action of the maximal torus $T^n$ in ${\rm GL}_N(\CC)$ on $\Lie {\rm GL}_N(\CC)/\Lie T^n$. Namely, in the torus-non-equivariant limit $J_{\Pi\mathfrak{g}/\mathfrak{t}}$, the ``small J-function'' $J^T_{\Pi \mathfrak{g}/\mathfrak{t}}$ introduced in Section~\ref{Sec4} satisfies (as it is not hard to check) the system of finite difference equations
\begin{gather*}
\prod_{i'\neq i}\big(1-q P_{i'}P_i^{-1}q^{Q_{i'}\p_{Q_{i'}}-Q_i\p_{Q_i}}\big)
\big(1-P_iq^{Q_i\p_{Q_i}}\big)^N J_{\Pi \mathfrak{g}/\mathfrak{t}}
\\ \qquad
{} = Q_i \prod_{i'\neq i} \big(1-q P_iP_{i'}^{-1} q^{Q_i\p_{Q_i}-Q_{i'}\p_{Q_{i'}}}\big) J_{\Pi \mathfrak{g}/\mathfrak{t}},\qquad i=1,\dots, n.
\end{gather*}
So, the claim is that our mirror $q$-integral satisfies the same system (for scalar-valued rather than $K^0(X)$-valued functions):
\begin{gather*}
 \prod_{i'\neq i}\big(1-q\ q^{Q_{i'}\p_{Q_{i'}}-Q_i\p_{Q_i}}\big) \big(1-q^{Q_i\p_{Q_i}}\big)^N \I
 = Q_i\prod_{i'\neq i} \big(1-q\ q^{Q_i\p_{Q_i}-Q_{i'}\p_{Q_{i'}}}\big)\I,\quad\ i=1,\dots, n.
\end{gather*}
To check this, we note that translation operators $q^{X_{ij}\p_{X_{ij}}}$ and $q^{-Y_{ii'}\p_{Y_{ii'}}}$ project to the $Q$-space into respectively $q^{Q_i\p_{Q_i}}$ and $q^{Q_i\p_{Q_i}\!-Q_{i'}\p_{Q_{i'}}}$\!. Applying $q^{X_{ij}\p_{X_{ij}}}$ to the factor $\Delta_{ij}\!:=\!{\rm e}^{\sum_{k>0}\!X_{ij}^k/k(1-q^k)}$ in the integrand of $\I$ containing $X_{ij}$, we obtain
\[
{\rm e}^{\sum_{k>0} q^kX_{ij}^k/k(1-q^k)}={\rm e}^{\sum_{k>0}X_{ij}^k/k(1-q^k)}{\rm e}^{-\sum_{k>0}X_{ij}^k/k}=(1-X_{ij}) \Delta_{ij}.
\]
Therefore, applying $\prod_j(1-q^{X_{ij}\p_{X_{ij}}})$, we find the integrand multiplied by $\prod_j X_{ij}$. Similarly, applying $q^{-Y_{ii'}\p_{Y_{ii'}}}$ to $\nabla_{ii'}:={\rm e}^{-\sum_{k>0}q^kY_{ii'}^{k}/k(1-q^k)}Y_{ii'}$,
we obtain $(1-Y_{ii'})q^{-1}\nabla_{ii'}$, and hence applying $1-q\ q^{-Y_{ii'}\p_{Y_{ii'}}}$ we find the integrand multiplied by $(1-q (1-Y_{ii'}) q^{-1})=Y_{ii'}$. Since $\prod_{i'\neq i}Y_{i'i}\prod_jX_{ij} =Q_i\prod_{i'\neq i} Y_{ii'}$ for $i=1,\dots, n$, the promised finite difference equations follow.

In this argument it was assumed that the family of $q$-integration
lattices $\Gamma\subset \X_{Q_1,\dots,Q_n}$ depending on $Q_1,\dots,Q_n$ was invariant
under all coordinate multiplicative $q$-translations in the ambient torus $\X$.
Also note that the same argument applies to the ordinary (as opposed to~Jackson's)
integrals, provided that (the families of) the cycles of integration are homologous to~their $q$-translates. The catch is that it is not entirely clear how to produce
such lattices and/or cycles. Below we will resolve this catch in the $T$-equivariant case.

The torus-equivariant counterpart $\I^T$ of $\I$ is obtained by inserting into the integrand the factor
\[
\prod_{i,j}X_{ij}^{\ln \Lambda_j/\ln q} \prod_{i\neq i'} Y_{ii'}^{\ln \Lambda_0/\ln q}.
\]
By repeating the above computations, we find that $\I^T$ satisfies finite difference equations
\begin{gather*}
\prod_{i'\neq i}\big(1-q \Lambda_0 q^{Q_{i'}\p_{Q_{i'}}-Q_i\p_{Q_i}}\big) \prod_j\big(1-q^{Q_i\p_{Q_i}}/\Lambda_j\big) \I^T
 = Q_i\prod_{i'\neq i} \big(1-q\Lambda_0 q^{Q_i\p_{Q_i}-Q_{i'}\p_{Q_{i'}}}\big) \I^T,
 \\
 \qquad i=1,\dots, n.
\end{gather*}
Replacing the natural action of finite difference operators on scalar-valued functions with the representation on $K^0_T(X)$-valued functions by $q^{Q_i\p_{Q_i}}\mapsto P_iq^{Q_i\p_{Q_i}}$, we obtain the equations satisfied by the series $J^T_{\Pi\mathfrak{g}/\mathfrak{t}}$ from Section~\ref{Sec4}.

Let us now examine $\I^T$ for a cycle $\Gamma$ fitting coordinate charts on $\X_{Q_1,\dots,Q_n}$. Picking an~injec\-tive function $J\colon \{ 1,\dots, n\} \to \{1,\dots, N\}$, we express $X_{iJ(i)}$, $i=1,\dots,n$, in terms of $Q_i$ and the remaining variables, using the equations of $\X_{Q_1,\dots,Q_n}$, and then rewrite the integral $\I^T$ in this chart. For instance, taking $J(i)=i$
we find $X_{ii}=Q_i\prod_{j\neq i}X_{ij}^{-1}\prod_{i'\neq i}(Y_{ii'}/Y_{i'i})$. Consequently
\[
\I^T_{(1,\dots,n)}=\sum_{0\leq d_1,\dots,d_n}\prod_i \frac{Q_i^{d_i+\ln \Lambda_i/\ln q}}{(1-q)(1-q^2)\cdots (1-q^{d_i})}\ \I^{(d_1,\dots,d_n)}_{(1,\dots,n)},
\]
where
\begin{gather*}
\I^{(d_1,\dots,d_n)}_{(1,\dots,n)}=\pm \int_{\Gamma_{(1,\dots,n)}} {\rm e}^{ \sum_{k>0} \left(\sum_{j\neq i}X_{ij}^k-q^k\sum_{i\neq i'}Y_{ii'}^k\right)/k(1-q^k)}
\\ \hphantom{\I^{(d_1,\dots,d_n)}_{(1,\dots,n)}=}
{}\times\prod_{i\neq j} X_{ij}^{-d_i+\ln (\Lambda_j/\Lambda_i)/\ln q} \prod_{i\neq i'} Y_{ii'}^{d_i-d_i'+1+\ln (\Lambda_0\Lambda_i/\Lambda_{i'})/\ln q} \!\bigwedge_{i\neq j} {\rm d}_q\ln X_{ij} \!\bigwedge_{i\neq i'} {\rm d}_q\ln Y_{ii'},
\end{gather*}
and the sign is determined by the order of the $q$-differentials and an orientation of the multiplicative $q$-lattice $\Gamma_{(1,\dots,n)}$. This is the product
of model $1$-dimensional $q$-integrals
\begin{gather*}
I^{(d)}_{+}:=\int_0^{\infty} {\rm e}^{ \sum_{k>0}X^k/k(1-q^k)} X^{-d+\ln (\Lambda'/\Lambda)/\ln q} {\rm d}_q\ln X,
\\
I^{(d)}_{-}:=\int_0^{-\infty} {\rm e}^{ -\sum_{k>0}q^kY^k/k(1-q^k)} Y^{1-d+\ln (\Lambda_0\Lambda'/\Lambda)/\ln q} {\rm d}_q\ln X,
\end{gather*}
considered at the beginning of this section. They can be expressed via $I_{\pm}^{(0)}$ by the recursive property of the $q$-gamma-like function. Explicitly, applying to the integrand of $I^{(d)}_{+}$ the operator $1-q^d q^{ X\p_X}\Lambda/\Lambda'$, we find (after a short computation) that $\big(1-q^d\Lambda/\Lambda'\big) I_{+}^{(d)} = I^{(d-1)}_{+}$. Applying this inductively we conclude that $I_{+}^{(d)} = I_{+}^{(0)}/\prod_{m=1}^d (1-q^m\Lambda/\Lambda')$. Using this and a similar recursion for $I^{(d)}_{-}$, we can reduce $\I^{(d_1,\dots,d_n)}_{(1,\dots,n)}$ to $\I^{(0,\dots,0)}_{(1,\dots,n)}$:
\[
\I^{(d_1,\dots,d_n)}_{(1,\dots,n)}=\frac{\I^{(0,\dots,0)}_{(1,\dots,n)}}
{ \prod_{i\neq j}\prod_{m=1}^{d_i}(1-q^m\Lambda_i/\Lambda_j)} \prod_{i\neq i'} \frac{\prod_{m=-\infty}^{d_i-d_i'}(1-q^m\Lambda_0\Lambda_i/\Lambda_{i'})}
{\prod_{m=-\infty}^0(1-q^m\Lambda_0\Lambda_i/\Lambda_{i'})}.
\]
Comparing this to the terms of the series $J^T_{\Pi\mathfrak{g}/\mathfrak{t}}$ from Section~\ref{Sec4} localized at the fixed point $(P_1,\dots,P_n)=(\Lambda_1,\dots,\Lambda_n)$, we arrive at
\[
\I^T_{(1,\dots,n)}=\big(J^T_{\Pi\mathfrak{g}/\mathfrak{t}}\big)_{(1,\dots,n)} \prod_i Q_i^{\ln \Lambda_i/\ln q} \I^{(0,\dots,0)}_{(1,\dots,n)}.
\]
Note that $\I^{(0,\dots,0)}_{(1,\dots,n)}$ doesn't depend on $Q$, while the role of the factor $Q_i^{\ln \Lambda_i/\ln q}$ is to conjugate~$q^{Q_i\p_{Q_i}}$ into $P_i q^{Q_i\p_{Q_i}}|_{P_i=\Lambda_i}$. Thus, one can say that $(\J^T_{\Pi\mathfrak{g}/\mathfrak{t}})_{(1,\dots,n)}$ is given by the $q$-integral
over the ``cycle'' which is the formal multiple of the $q$-lattice $\Gamma_{(1,\dots,n)}$ (lifted from our chart to the torus $\X_{Q_1,\dots,Q_n}$) with the coefficient inverse to
$\prod_i Q_i^{\ln \Lambda_i/\ln q} \I^{(0,\dots,0)}_{(1,\dots,n)}$.

When the indexing function $J\colon\{1,\dots,n\}\to \{1,\dots,N\}$ is changed to
a permutation $\sigma$ of~$\{ 1,\dots, n\}$, the value of the $q$-integral does not change
$\I^T_{\sigma(1),\dots,\sigma(n)}=\I^T_{1,\dots,n}$, assuming that the order the orientation
of the $q$-lattice $\Gamma_{(\sigma(1),\dots,\sigma(n))}$ is consistent with the order of the
$q$-differentials in the wedge-product.

Thus, the above computation extended to arbitrary injective indexing functions $J$ shows that all relevant components of the vector-function $J^T_{\Pi \mathfrak{g}/\mathfrak{t}}$ can be represented by our $q$-integrals using ``cycles'' fitting appropriate charts.
Setting in the $q$-integral $Q_1=\cdots=Q_n=Q$ and $\Lambda_0=1$, we obtain a torus-equivariant K-theoretic mirror of the grassmannian.

\begin{Theorem}\label{thm7}
The multi-dimensional $q$-integral
\begin{align*}
\I_X^T :=\int_{\Gamma\subset \X_{Q}} &{\rm e}^{ \sum_{k>0}\left(\sum_{i,j}X_{ij}^k - q^k \sum_{i\neq i'}Y_{ii'}^k\right)/k(1-q^k)} \\
&\times \prod_{i,j}X_{ij}^{\ln \Lambda_j/\ln q} \prod_{i\neq i'} Y_{ii'} \frac{\bigwedge_i \big(\bigwedge_j {\rm d}_q\ln X_{ij} \bigwedge_{i'\neq i}{\rm d}_q \ln Y_{ii'}\big)}{\bigwedge_i \big( \sum_j{\rm d}_q\ln X_{ij} -\sum_{i'\neq i}{\rm d}_q\ln (Y_{ii'}/Y_{i'i})\big)},
\end{align*}
with suitable choices of (linear combinations of) $q$-lattices $\Gamma$ in
\[
\X_{Q}:=\bigg\{ (X,Y)\in \X \bigm| \prod_jX_{ij}=Q \prod_{i'\neq i}(Y_{ii'}/Y_{i'i}),\
i=1,\dots, n \bigg\}
\]
represents components of the $K^0_T(X)$-valued small J-function $J^T_X$ of the grassmannian \mbox{$X\!=\!{\rm Gr}_{n,N}$}.
\end{Theorem}

\subsection*{Acknowledgments}

This material is based upon work supported by the National Science Foundation under Grant DMS-1906326.
We are thankful to P.~Koroteev and A.~Smirnov for their effort in educating us about their work on quantum K-theory of symplectic quiver varieties, and to H.~Liu and Y.~Wen for sharing and discussing their preprints.

\pdfbookmark[1]{References}{ref}
\LastPageEnding

\end{document}